\documentclass[preprint,12pt]{elsarticle}
\usepackage{indentfirst}
\usepackage{amssymb}
\usepackage{verbatim}
\usepackage{amsthm}
\usepackage{fleqn}
\usepackage{graphicx}
\usepackage{epstopdf}
\usepackage{amsfonts}
\usepackage[colorlinks,citecolor=blue,urlcolor=blue]{hyperref}
\usepackage{amsmath}
\usepackage{bm}
\usepackage{appendix} 
\usepackage{geometry}
\usepackage{pifont}
\usepackage{mathrsfs}
\usepackage{tcolorbox}
\tcbuselibrary{breakable, skins}
\usepackage{graphicx} 
\usepackage{booktabs}

\usepackage{listings}
\usepackage{algorithm}   
\usepackage{algpseudocode} 
\usepackage{algpseudocodex} 
\newtheorem{theorem}{Theorem}[section]
\newtheorem{lemma}[theorem]{Lemma}
\newtheorem{proposition}[theorem]{Proposition}

\newtheorem{remark}[theorem]{Remark}

\newtheorem{assumption}[theorem]{Assumption}

\newtcolorbox{bluepar}{
	colback=white,           
	colframe=blue!60!black,  
	breakable,
	enhanced,
	boxrule=1.2pt,           
	arc=4pt,
	left=8pt,
	right=8pt,
	top=8pt,
	bottom=8pt,
	before skip=10pt,        
	after skip=10pt          
}

\newcommand{\cA}{\mathcal{A}}
\newcommand{\cB}{\mathcal{B}}
\newcommand{\cC}{\mathcal{C}}
\newcommand{\cD}{\mathcal{D}}

\newcommand{\cF}{\mathcal{F}}

\newcommand{\cH}{\mathcal{H}}

\newcommand{\cK}{\mathcal{K}}
\newcommand{\cL}{\mathcal{L}}
\newcommand{\cM}{\mathcal{M}}
\newcommand{\cN}{\mathcal{N}}
\newcommand{\cO}{\mathcal{O}}
\newcommand{\cP}{\mathcal{P}}

\newcommand{\cS}{\mathcal{S}}
\newcommand{\cT}{\mathcal{T}}
\newcommand{\cU}{\mathcal{U}}

\newcommand{\cW}{\mathcal{W}}

\newcommand{\EE}{\mathbb{E}}
\newcommand{\FF}{\mathbb{F}}

\newcommand{\LL}{\mathbb{L}}
\newcommand{\RR}{\mathbb{R}}

\newcommand{\bu}{\mathbf{u}}
\newcommand{\ti}{\tilde}

\newcommand{\al}{\alpha}

\newcommand{\ep}{\epsilon}

\newcommand{\la}{\lambda}

\newcommand{\drm}{\mathrm{d}}
\newcommand{\tu}{\tilde{\tilde{u}}}

\newcommand{\bea}{\bm{\eta}}
\newcommand{\bl}{\bm{\lambda}}

\usepackage{float}  
\usetikzlibrary{positioning}
\numberwithin{equation}{section}

\allowdisplaybreaks[4]

\begin{document}

\begin{frontmatter}

\title{Stochastic Mean-Field LQ Stackelberg Differential Games with Random Coefficients: Theory and a Deep FBSDE Picard Solver}

\author{Ying Yang\corref{cor1}}
\address{Department of Mathematics, Southern University of Science and Technology, Shenzhen 518055, P. R. China}
\ead{12331007@mail.sustech.edu.cn}

\author{Jie Xiong}
\address{Department of Mathematics and SUSTech International Center for Mathematics, Southern University of Science and Technology, Shenzhen 518055, P. R. China}
\ead{jxiong@mail.sustech.edu.cn}

\author{Zhouyu Wang}
\address{Department of Mathematics, Southern University of Science and Technology, Shenzhen 518055, P. R. China}
\ead{12431014@mail.sustech.edu.cn}

\cortext[cor1]{Corresponding author. \\ 
\quad*This work was supported by the National Key R\&D Program of China (2022YFA1006102), and the National Natural Science Foundation of China (12471418).}

 \begin{abstract}
This paper studies a stochastic mean-field linear-quadratic Stackelberg differential game with random coefficients. The interaction between mean-field terms and random coefficients precludes the direct use of conventional decoupling techniques. We apply an extended Lagrange multiplier method to derive an affine operator representation of the follower's optimal response. The induced leader problem is then formulated as a generalized
stochastic LQ control problem with operator-valued coefficients, and the Stackelberg optimal control is characterized through a Riccati-free coupled FBSDE system. We further develop a Deep FBSDE Picard Solver that preserves the Stackelberg order through follower-response learning, response-sensitivity extraction,
leader optimization, and neural augmented Lagrangian enforcement of mean-field consistency constraints. Numerical studies covering convergence diagnostics, discretization sensitivity, Riccati calibration, ablation tests, stability under control perturbations, Stackelberg--Nash comparisons, and a financial application support the effectiveness of the proposed framework.
 \end{abstract}


\end{frontmatter}
\section{Introduction} 
Dynamic games provide a natural framework for hierarchical optimization problems involving multiple decision makers with asymmetric roles. The Stackelberg game \cite{von2010market} is the classical model for such leader--follower interaction: the leader commits to a strategy first, and the follower responds optimally. This forces the leader to solve a fundamentally bilevel problem that is substantially harder than single-agent optimal control.  This structure arises in regulation, contract design, and resource allocation—situations where a dominant agent must anticipate the rational behavior of subordinates before acting. In continuous-time stochastic settings, linear-quadratic (LQ) models are particularly important because their linear state dynamics and quadratic performance criteria provide analytical tractability while remaining sufficiently expressive for applications such as portfolio allocation, production planning, resource regulation, and risk-sensitive tracking.

Classical deterministic LQ Stackelberg games and their open-loop equilibria have been extensively investigated; see, for example, \cite{abou1985analytical, freiling2001existence}. For stochastic systems, Yong \cite{yong2002leader} showed that incorporating the follower's rational response transforms the leader's problem into a stochastic control problem constrained by a forward-backward stochastic differential equation (FBSDE), revealing a fundamental difficulty absent in the single-agent case. This framework has since been extended to settings such as  partial information, mean-field interactions, jump diffusions, regime switching, and infinite-dimensional systems; see, for example,\cite{shi2016leader,moon2021linear,li2021linear,lv2020two} and \cite{ding2025infinite}.

Beyond classical LQ Stackelberg games, mean-field formulations provide a convenient mechanism for capturing aggregate effects in large-scale interacting systems. In LQ models, such effects are typically represented by the expectations of the state and control processes in the dynamics and cost functionals. Related mean-field backward stochastic equations (SDEs), mean-field forward–backward SDEs (FBSDEs), optimal control problems have been extensively studied; see, e.g.,
\cite{buckdahn2009mean},\cite{carmona2013mean},   \cite{carmona2018probabilistic}. These ideas have also been incorporated into hierarchical decision-making, leading to mean-field Stackelberg games; see, for example, \cite{bensoussan2015mean},\cite{lin2018open},\cite{lv2023linear},  \cite{wang2025linear}.

However, most existing works either focus on large-population decentralized equilibria or rely on deterministic or specially structured coefficients. Such restrictions are often inadequate in financial and engineering applications where model coefficients evolve randomly with the underlying information flow.  This paper investigates a random-coefficient mean-field LQ Stackelberg problem. A key mathematical difficulty in mean-field stochastic LQ control with random coefficients is that the adjoint equations may involve cross-moment terms such as $\mathbb{E}\big[A(t)^\top Y(t)\big]$, which in general cannot be simplified into $\mathbb{E}\big[A(t)^\top\big]\mathbb{E}[Y(t)]$. To overcome this difficulty, Xiong and Xu \cite{xiong2025mean} developed an extended Lagrange multiplier method, which introduces auxiliary deterministic variables for the mean state and mean control and relaxes the resulting consistency constraints through extended Lagrange multipliers.

Inspired by this approach, we adapt the extended Lagrange multiplier method to the hierarchical structure of a Stackelberg game. For a fixed leader's control, we first solve the follower's random-coefficient mean-field LQ problem and show that the follower's optimal response admits an affine operator representation
with respect to the initial state $x$, the leader's control $u_2(\cdot)$, and an inhomogeneous term. When this response is substituted into the leader's dynamics, the induced leader problem is governed by stochastic operator-valued coefficients. In contrast to \cite{wei2019linear}, where optimal controls for operator-valued LQ systems are characterized through integral kernel representations, our extended Lagrange multiplier approach produces an affine response-operator structure that can be naturally incorporated into the Deep FBSDE Picard Solver (DFPS) developed  in
Section~\ref{sec:5}.

By invoking the stochastic maximum principle, the optimal solution to this generalized problem can be characterized by a deeply coupled FBSDE system. However, solving this resulting system numerically remains highly nontrivial. Classical Riccati-based approaches become difficult to apply under stochastic operator-valued coefficients, while direct numerical discretization suffers from the severe coupling among the forward state equation, the backward adjoint equations, the mean-field consistency constraints, and the bilevel dependence of the follower's response on the leader's control.

Deep learning methods have recently provided powerful tools for
high-dimensional stochastic control problems and FBSDEs. Han et al.
\cite{han2017deep,han2018solving} proposed the deep BSDE method, which represents the martingale integrand by neural networks and trains the unknown initial value through a terminal loss; see also
\cite{beck2019machine,hu2019deep,han2020solving} for further developments. For fully coupled FBSDEs, Han and Long \cite{han2020convergence} established convergence guarantees under neural-network approximation. Ji et al. \cite{ji2022deep} further reformulated a fully coupled FBSDE as a stochastic Stackelberg differential game and solved it through a bi-level deep learning
procedure. These works provide important numerical tools for high-dimensional FBSDEs and stochastic control problems. However, they are not designed for the response-induced stochastic operator-valued coefficients arising in the present mean-field Stackelberg problem, where the leader's dynamics can only be formed
after the follower's rational response has been characterized.

Motivated by this theoretical structure, we propose a Deep FBSDE Picard Solver (DFPS) tailored to the operator-valued mean-field Stackelberg system. Rather than treating the bilevel game as a simultaneous system, DFPS preserves the Stackelberg order through a sequential pipeline: follower-response learning, response-sensitivity extraction, and leader optimization. A key difficulty is that the mean-field quantities $\mathbb E[X(t)]$ and $\mathbb E[u_i(t)]$ are endogenous equilibrium objects rather than exogenous coefficients. Hence, a direct Monte Carlo plug-in treatment would externalize these processes as batch-wise sample statistics and does not by itself enforce the mean-field
fixed-point consistency. Within DFPS, Picard iterations are used internally to handle the forward--backward coupling and mean-field consistency constraints in the corresponding player-specific FBSDE systems. Specifically, DFPS  trains the follower's response under exploratory leader controls, with mean-field consistency constraints enforced through neural augmented Lagrangian updates. It then extracts the follower's affine response sensitivities with respect to the leader's control, and finally trains the leader's policy using the follower-induced dynamics and the extracted bilevel sensitivities.
 
The main contributions of this paper are summarized as follows:
\begin{itemize}
 \item We adapt the extended Lagrange multiplier method to the Stackelberg hierarchy and obtain an affine operator representation of the follower's 	rational response. This representation characterizes how the follower's response induces a generalized leader problem governed by stochastic operator-valued coefficients.
	
\item We propose a Deep FBSDE Picard Solver (DFPS) that preserves the Stackelberg order through a sequential pipeline of follower-response learning, response-sensitivity extraction, and leader optimization. The mean-field quantities $\mathbb{E}[X(t)]$ and $\mathbb{E}[u_i(t)]$ are endogenous equilibrium objects rather than exogenous coefficients, so DFPS uses Picard iterations to handle the forward--backward coupling and enforces mean-field consistency through neural augmented Lagrangian update mechanism. Numerical experiments illustrate convergence, structural component necessity, 
and numerical stability under control perturbations.
\end{itemize}

The remainder of this paper is organized as follows. Section~\ref{sec1} introduces the stochastic mean-field LQ Stackelberg model, defines the admissible control spaces, and provides several preliminary estimates. Section~\ref{sec2} studies the follower's and the leader's problems, deriving the associated optimality conditions and affine response representations. Section~\ref{sec:5} details DFPS and its augmented Lagrangian implementation, and demonstrates its performance through numerical convergence and feasibility tests, discretization sensitivity analysis, a Riccati sanity check, ablation studies, equilibrium validation, and a financial application. Section~\ref{sec:conclusion} concludes the paper.

\section{Model and Preliminaries}\label{sec1}
\subsection{Model}
Let $(\Omega, \mathcal{F}, \mathbb{F}, \mathbb{P})$ be a complete filtered probability space on which a one-dimensional standard Brownian motion $\{W(t): 0 \leq t \leq T\}$ is defined. Here, $\mathbb{F} = \{\mathcal{F}_t\}_{t \geq 0}$ denotes the natural filtration generated by $W(t)$, augmented by all $\mathbb{P}$-null sets. We first consider   the following controlled linear forward stochastic differential equation with random coefficients on the time interval $[0,T]$:
\begin{equation}\label{state0}
	\left\{\begin{aligned}
		\drm X(t)=&\left[A_1(t)X(t)+A_2(t)\mathbb{E}[X(t)]+B_1(t)u_1(t)+B_2(t)u_2(t)+b(t)\right]\drm t\\
		&+\left[C_1(t)X(t)+C_2(t)\mathbb{E}[X(t)]+D_1(t)u_1(t)+D_2(t)u_2(t)+\sigma(t)\right]\drm W(t),\\
		X(0)=&x,
	\end{aligned}\right.
\end{equation}
where $A_i(\cdot), C_i(\cdot):[0,T]\times \Omega \to \mathbb{R}^{n\times n}$ and $ B_i(\cdot),   D_i :[0,T]\times \Omega \to  \mathbb{R}^{n\times m_i}$ for \(i=1,2\),  are matrix-valued $\mathbb{F}$-adapted processes, and $b(\cdot), \sigma(\cdot):[0,T]\times \Omega\to\RR^n$ are $\FF$-adapted processes as inhomogeneous terms.   The initial state \(x\in\mathbb{R}^n\) is fixed throughout this article. And $X(\cdot)$ valued in $\mathbb{R}^{n}  $ is the state process. Moreover, $u_i(\cdot)$ are valued in $\mathbb{R}^{m_i}$ (for $i=1,2$) and are $\mathbb{F}$-adapted processes satisfying $\mathbb{E}\left[\int_0^T|u_i(s)|^2ds\right] < \infty$, which represent the control processes of the follower (for $i=1$) and the leader (for $i=2$). For notational simplicity, we further denote \(\bar{X}\) as the expectation of \(X\), and \(\bar{u}_i\) as the expectation of \(u_i\) (for \(i=1,2\)), i.e., \(\bar{X}(\cdot) = \mathbb{E}[X(\cdot)]\), and \(\bar{u}_i(\cdot)=\EE[u_i(\cdot)]\).

Now the follower and the leader seek to minimize the following objective functional with random coefficients for each player
\begin{equation}\label{cost0}
	\begin{aligned}
		J_i( u_1(\cdot),u_2(\cdot))=&\mathbb{E}\Big\{\int_0^T\left[\langle Q_i(s)X(s),X(s)\rangle+\langle\bar{Q}_i(s)\bar{X}(s),\bar{X}(s)\rangle+ \langle R_i(s)u_i(s),u_i(s)\rangle\right. \\
		&\left.\quad\qquad +\langle \bar{R}_i(s)\bar{u}_i(s),\bar{u}_i(s)\rangle\right]\drm s+\langle G_iX(T),X(T)\rangle\Big\},\qquad \qquad i=1,2,
	\end{aligned}
\end{equation}
where for $i=1,2$, $G_i$ are $\mathcal{F}_T$-measurable random matrices, $Q_i(\cdot), \bar{Q}_i(\cdot):[0,T]\times \Omega \to \mathbb{R}^{n\times n}$, and $R_i(\cdot), \bar{R}_i(\cdot):[0,T]\times\Omega\to\mathbb{R}^{m_i\times m_i}$.

We assume that the admissible control sets for the player \(i\) (for \(i=1,2\)) are defined as follows:
\[
\mathcal{U}_i[0, T] = \left\{ u_i(\cdot) : [0, T] \times \Omega \to \mathbb{R}^{m_i} \mid u_i(\cdot) \text{ is $\mathcal{F}_t$-adapted and } \mathbb{E} \int_0^T |u_i(t)|^2 dt < \infty \right\}, \quad i = 1, 2.
\]

 With these control sets defined, the problem is formally formulated as a Linear-Quadratic Mean-Field Stackelberg Differential Game with random coefficients. To tackle the inherent complexity of this game, our solution strategy builds upon the general Stackelberg framework introduced by Yong \cite{yong2002leader}, adapting it to incorporate the techniques for random-coefficient mean-field LQ problems developed by Xiong and Xu \cite{xiong2025mean}. Specifically, this integrated approach naturally unfolds in the following two steps:
\begin{description}
	\item[Step 1: Solving the follower's problem.]
	For any given admissible control $u_2(\cdot) \in \mathcal{U}_2[0,T]$ of the leader and a fixed initial state $x \in \mathbb{R}^n$, the follower's problem constitutes a mean-field stochastic linear-quadratic (MFSLQ) control problem with random coefficients. Following \cite{xiong2025mean}, we introduce the auxiliary variables $\bar{u}_1(\cdot) = \alpha_1(\cdot)$ and $\bar{X}(\cdot) = \beta_1(\cdot)$ to recast it as a constrained optimization problem, which is then relaxed via the Extended Lagrange Multipliers (ELMs) method. This allows us to tackle a general linear-quadratic problem with respect to $\alpha_1(\cdot)$ and $\beta_1(\cdot)$, ultimately yielding an affine representation of the follower's optimal response $\tilde{u}_1(\cdot)$. In this representation, the associated operators act linearly on the initial state $x$ and the leader's control $u_2(\cdot)$.
	
	\item[Step 2: Solving the leader's problem.]
	By substituting the follower's affine response $\tilde{u}_1(\cdot)$ back into the state dynamics, the leader's problem is transformed into a generalized MFSLQ control problem. The leader's optimal strategy $\tilde{u}_2(\cdot)$ is subsequently characterized by deriving the corresponding optimality system, employing the same ELMs methodology utilized in Step 1.
\end{description}

To formulate the leader's problem, we substitute the affine operator representation of $\tilde{u}_1(\cdot)$ into the original state equation \eqref{state0}. This substitution naturally gives rise to a generalized state equation governed by operator-valued stochastic processes, which we define as follows:
\begin{equation}\label{state1}
	\left\{
	\begin{aligned}
		\drm X(t) &= \left[ (\cA_1 X)(t) + (\cA_2 \bar{X})(t)+(\cB_1 u_1)(t) + (\cB_2 u_2)(t) + b(t) \right]\drm t \\
		&\quad + \left[ (\cC_1 X)(t)+(\cC_2\bar{X})(t) + (\cD_1 u_1)(t) + (\cD_2 u_2)(t) + \sigma(t) \right]\drm W(t), \\
		X(0) &= x,
	\end{aligned}
	\right.
\end{equation}
where $\cA_i(\cdot), \cB_i(\cdot), \cC_i(\cdot)$, and $\cD_i(\cdot)$ (for $i=1,2$) are suitably defined bounded linear operators. Furthermore, the inhomogeneous terms $b(\cdot)$ and $\sigma(\cdot)$ are $\mathbb{F}$-adapted square-integrable stochastic processes.

With these generalized dynamics established, we can now formally characterize the optimal control problems for both the follower and the leader.

\noindent\textbf{Problem (MFSOLQ-F).} For a given initial state $x \in \mathbb{R}^n$ and any fixed leader's control $u_2(\cdot) \in \mathcal{U}_2[0,T]$, find a control $\tilde{u}_1(\cdot) \in \mathcal{U}_1[0,T]$ that minimizes the cost functional \eqref{cost0} for $i=1$ subject to the state equation \eqref{state1}, i.e.,
\begin{align}
	J_1\left( \tilde{u}_1(\cdot), u_2(\cdot)\right) = \inf_{u_1(\cdot) \in \mathcal{U}_1} J_1\left( u_1(\cdot), u_2(\cdot)\right).
\end{align}

As previously discussed, the follower's optimal control $\tilde{u}_1(\cdot)$ generally depends on both the initial state $x$ and the leader's control strategy $u_2(\cdot)$. To reflect this dependence explicitly, we adopt the notation $\tilde{u}_1[x, u_2](\cdot)$ to denote the follower's optimal response function.

\noindent\textbf{Problem (MFSOLQ-L).} Given the follower's optimal response $\tilde{u}_1[x, u_2(\cdot)](\cdot)$, find a control $\tilde{u}_2 (\cdot) \in \mathcal{U}_2[0,T]$ that minimizes the cost functional \eqref{cost0} for $i=2$, i.e., 
\begin{align}
	J_2\left( \tilde{u}_1[x, \tilde{u}_2](\cdot), \tilde{u}_2(\cdot)\right) = \inf_{u_2(\cdot) \in \mathcal{U}_2} J_2\left( \tilde{u}_1[x, u_2(\cdot)](\cdot), u_2(\cdot)\right).
\end{align}

\subsection{Preliminaries}\label{sec:2}
For a random variable $\xi$, we write $\xi \in \mathcal{F}_t$ if $\xi$ is $\mathcal{F}_t$-measurable; for a stochastic process $\phi(\cdot)$, $\phi(\cdot) \in \mathbb{F}$ means it is $\mathbb{F}$-adapted. For Euclidean spaces $\mathbb{H} = \mathbb{R}^n, \mathbb{R}^{m\times n}, \mathbb{S}^n_+$, and $p,q > 0$, we define the following spaces:
\begin{itemize}
	\item $L^{p,q}_\mathbb{F}(\mathbb{H})\equiv L^p_\mathbb{F}(\Omega;L^q([0,T];\mathbb{H}))$: the space of $\mathbb{F}$-adapted processes $X:[0,T]\times\Omega\rightarrow\mathbb{H}$ such that $\mathbb{E}\Big[\Big(\int_0^T \|X(s,\omega)\|_\mathbb{H}^q \mathrm{d}s\Big)^{p}\Big]<\infty$.
	\item $L^{2,c}_{\mathbb{F}}(\mathbb{R}^n)\equiv L_{\mathbb{F}}^{2,c}\big(\Omega; C([0,T]; \mathbb{H}) \big)$: the space of continuous $\mathbb{F}$-adapted processes $X: [0, T]\times \Omega \to \mathbb{H}$ such that $\mathbb{E}\left[\sup_{0 \leq s \leq T} \|X(s,\omega)\|_\mathbb{H} ^2\right]< \infty$.
	\item $L^2_{\mathbb{F}}(\mathbb{H})\equiv L_{\mathbb{F}}^{2}(0, T; \mathbb{H})$: the space of $\mathbb{F}$-adapted $\mathbb{H}$-valued square-integrable stochastic processes.
	\item $L^{\infty,c}_{\mathbb{F}}(\mathbb{H})\equiv L^{\infty,c}_\mathbb{F} (0,T;\mathbb{H})$: the space of $\mathbb{F}$-adapted $\mathbb{H}$-valued bounded continuous processes.
	\item $L_{\mathcal{G}}^{2}(\mathbb{H})\equiv L^2_\mathcal{G}(\Omega; \mathbb{H})$: the space of $\mathcal{G}$-measurable $\mathbb{H}$-valued square-integrable random variables, where $\mathcal{G}\subset \mathbb{F}$ is a sub-$\sigma$-field.
	\item $L^\infty_\mathcal{G}(\mathbb{H})\equiv L^\infty_\mathcal{G}(\Omega;\mathbb{H})$: the space of $\mathcal{G}$-measurable $\mathbb{H}$-valued bounded random variables.
	\item $\mathbb{L}^2$: the space of deterministic, real-valued, square-integrable functions on $[0, T]$.
	\item $\mathcal{L}^p_{\mathbb{F}}(\mathcal{X};\mathcal{Y})$ (for $p\in[0,\infty]$): the space of all $\mathbb{F}$-adapted operators $\mathcal{B}(t)\colon \mathcal{X}\to\mathcal{Y}$, equipped with the norm
	\[
	\lVert \mathcal{B}(\cdot) \rVert_p
	=
	\begin{cases}
		\displaystyle\left( \int_0^T \lVert \mathcal{B}(t) \rVert^p \mathrm{d} t \right)^{\frac{1}{p}}, & p \in [1,\infty), \\[4pt]
		\operatorname*{ess\,sup}_{t \in [0,T]} \lVert \mathcal{B}(t) \rVert, & p = \infty,
	\end{cases}
	\]
	where the operator norm $\|\mathcal{B}(t)\|$ is defined by
	\[
	\|\mathcal{B}(t)\|
	=
	\sup\!\left\{
	\big(\mathbb{E}[|\mathcal{B}(t)\eta|^2]\big)^{\frac{1}{2}}
	\;\big|\;
	\eta\in L^2_{\mathcal{F}_t}(\mathbb{R}^m),\
	\big(\mathbb{E} [|\eta|^2]\big)^{\frac{1}{2}}=1
	\right\}.
	\]
	In particular, we denote $\mathcal{L}^p_{\FF}(\mathcal{X})=\mathcal{L}^p_{\FF}(\mathcal{X};\mathcal{X})$ for $p\in[0,\infty]$.
\end{itemize}

We introduce the following standard assumptions:

\textbf{(H1):} $\mathcal{A}_1,\mathcal{C}_1\in \mathcal{L}^\infty_{\mathbb{F}}\left(L^{2}_{\mathcal{F}_T}(\mathbb{R}^n)\right)$, $\mathcal{A}_2,\mathcal{C}_2\in\mathcal{L}^\infty_{\mathbb{F}}(\mathbb{L}^2; L^2_{\mathcal{F}_T}(\mathbb{R}^n))$, $b(\cdot),\sigma(\cdot)\in L^{2}_{\mathbb{F}}(\mathbb{R}^n)$, and $\mathcal{B}_i,\mathcal{D}_i\in \mathcal{L}^{\infty}_{\mathbb{F}}\left(L^2_{\mathcal{F}_T}(\mathbb{R}^{m_i});L^2_{\mathcal{F}_T}(\mathbb{R}^n)\right)$ for $i=1,2$.

\textbf{(H2):} $G_i \in L^\infty_{\mathcal{F}_T}(\mathbb{S}_+^n)$; $Q_i(\cdot), \bar{Q}_i(\cdot)\in L^\infty_\mathbb{F}(\mathbb{S}^n_+)$; $R_i(\cdot), \bar{R}_i(\cdot)\in L^\infty_\mathbb{F}(\mathbb{S}^{m_i}_+)$ for $i=1,2$. Moreover, there exists a constant $\delta>0$ such that $\bar{Q}_i(s)\geq \delta \mathbf{I}_{n}$ and $R_i(s), \bar{R}_i(s) \geq \delta \mathbf{I}_{m_i}$ a.e.\ $s\in[0,T]$, a.s.

The following lemma provides standard a priori estimates for the state process, which are essential for establishing the well-posedness of our optimal control problems. Its proof follows the standard arguments for general control problems and is thus omitted here.

\begin{lemma}\label{lem:4.1}
	Let (H1) hold. Then, for any $x\in\mathbb{R}^n$ and any pair of controls $(u_1(\cdot),u_2(\cdot))\in\mathcal{U}_1\times\mathcal{U}_2$, the state equation \eqref{state1} admits a unique adapted solution $X(\cdot)\in L^{2,c}_{\mathbb{F}}(\mathbb{R}^n)$. Moreover, there exists a constant $K>0$, independent of $x$, $u_1(\cdot)$, and $u_2(\cdot)$, such that 
	\begin{align}\label{eq:ieX}
		\mathbb{E}\left[ \sup_{0 \leq t \leq T} \lvert X(t) \rvert^2 \right] \leq K \left( \lvert x \rvert^2 + \mathbb{E}\int_0^T \Big( \lvert u_1(s) \rvert^2 + \lvert u_2(s) \rvert^2 + \lvert b(s) \rvert^2 + \lvert \sigma(s) \rvert^2 \Big) \mathrm{d}s \right),
	\end{align}
	where the constant \(K>0\) depending on \(\|\cA_i(\cdot)\|_\infty\), \(\|\cB_i(\cdot)\|_\infty\), \(\|\cC_i(\cdot)\|_\infty\), and \(\|\cD_i(\cdot)\|_\infty\) for \(i=1,2\).
\end{lemma}

Next, adapted from \cite[Proposition 2.6]{wei2019linear}, we state the well-posedness and regularity results for the associated backward stochastic differential equation (BSDE):
\begin{equation}\label{eq:BSDE1}
	\left\{
	\begin{aligned}
		\mathrm{d} Y(s)&=-\left[\mathcal{A}_1^\ast Y+\mathcal{C}_1^\ast Z+\mathbb{E}\left[\mathcal{A}_2^\ast Y+\mathcal{C}_2^\ast Z\right]+Q\right]\mathrm{d} s+Z(s)\mathrm{d} W(s),\\
		Y(T)&=\zeta\in L_{\mathcal{F}_T}^2(\mathbb{R}^n).
	\end{aligned}
	\right.
\end{equation}

\begin{lemma}\label{lem:esYZ}
	Let (H1) and (H2) hold, and suppose $Q(\cdot)\in L^{2}_{\mathbb{F}}(\mathbb{R}^n)$. Then, the BSDE \eqref{eq:BSDE1} admits a unique solution $(Y(\cdot),Z(\cdot))\in L^{2,c}_{\mathbb{F}}(\mathbb{R}^n)\times L^2_{\mathbb{F}}(\mathbb{R}^n)$. Moreover,
	\begin{align}
		\mathbb{E}\left[\sup_{0\leq t\leq T}|Y(t)|^2\right]+\mathbb{E} \left[\int_0^T |Z(s)|^2\mathrm{d} s\right]\leq K\mathbb{E}\left[|\zeta|^2+\int_0^T |Q(s)|^2\mathrm{d} s\right],
	\end{align}
	where $K>0$ is a constant depending on $\|\mathcal{A}_i(\cdot)\|_\infty$ and $\|\mathcal{C}_i(\cdot)\|_\infty$ for $i=1,2$.
\end{lemma}

\section{Theory of the Problem (MFSOLQ-F) and the Problem (MFSOLQ-L)}\label{sec2}
\subsection{Solving the Problem (MFSOLQ-F)}\label{sec:3}

In this subsection, we detail the solution procedure for Problem (MFSOLQ-F). Since the leader's control $u_2(\cdot) \in \mathcal{U}_2[0,T]$ is fixed throughout this stage, the terms involving $u_2(\cdot)$ act as given exogenous processes and can naturally be absorbed into the inhomogeneous terms of the system. Due to space limitations, the detailed proofs of the theoretical results presented in this subsection are deferred to   \ref{app:2}.

To address Problem (MFSOLQ-F), we first establish the strict convexity of the cost functional $J_1(u_1(\cdot),u_2(\cdot))$ with respect to the follower's control. This structural property is crucial, as it ensures the existence and uniqueness of the optimal response $\tilde{u}_1(\cdot)$. 

\begin{theorem}\label{Th:3.1}
	Let (H1) and (H2) hold. Then, for any initial state $x\in\mathbb{R}^n$ and any fixed leader's control $u_2(\cdot)\in\mathcal{U}_2[0,T]$, the cost functional $J_1(u_1(\cdot),u_2(\cdot))$ is strictly convex with respect to $u_1(\cdot)$.
\end{theorem}

\begin{theorem}\label{Th:3.2}
	Suppose (H1) and (H2) hold. Then, for any fixed $u_2(\cdot)\in\mathcal{U}_2[0,T]$ and $x\in\mathbb{R}^n$, Problem (MFSOLQ-F) admits a unique optimal control $\tilde{u}_1(\cdot)$. Moreover, a control $\tilde{u}_1(\cdot)$ is optimal if and only if the adapted solution $(\tilde{X}(\cdot), \tilde{Y}(\cdot), \tilde{Z}(\cdot))$ to the following coupled forward-backward stochastic differential equation (FBSDE):
	\begin{equation}\label{eq:F1}
		\left\{
		\begin{aligned}
			\drm\tilde{X}(s) &= \left[ \cA_1 \tilde{X} + \cA_2 \bar{X} + \cB_1 \tilde{u}_1 + \cB_2 u_2+b\right] \mathrm{d}s  + \left[ \cC_1 \tilde{X} + \cC_2 \bar{X} + \cD_1 \tilde{u}_1 + \cD_2 u_2 +\sigma \right] \mathrm{d}W(s), \\
			\mathrm{d}\tilde{Y}(s) &= -\left[ \cA_1^\ast \tilde{Y} + \cC_1^\ast \tilde{Z} + Q_1 \tilde{X} + \mathbb{E}[\bar{Q}_1] \bar{X} + \mathbb{E}\left[ \cA_2^\ast \tilde{Y} + \cC_2^\ast \tilde{Z} \right]\right] \mathrm{d}s+ \tilde{Z} \mathrm{d}W(s), \\
			\tilde{X}(0) &= x, \quad \tilde{Y}(T) = G_1 \tilde{X}(T).
		\end{aligned}
		\right.
	\end{equation}
	satisfies the stationarity condition:
	\begin{equation}\label{eq:so}
		R_1  \tilde{u}_1  + \cB_1^\ast  \tilde{Y}  + \cD_1^\ast \tilde{Z} + \mathbb{E}[\bar{R}_1] \mathbb{E}[\tilde{u}_1] = 0.
	\end{equation}
\end{theorem}

Here, and in what follows, we suppress the explicit dependence on the time variable $s$ for notational brevity whenever no confusion arises.

Since the parameters are stochastic operator-valued processes, they are not necessarily independent of \((X(\cdot),Y(\cdot),Z(\cdot))\).  Consequently, the decoupling of FBSDE \eqref{eq:F1} becomes challenging. To address this issue, we adopt a method inspired by \cite{xiong2025mean}, which transforms Problem (MFSOLQ-F) into a constrained control problem with constraints \(\bar{u}_1(\cdot) =\alpha_1(\cdot)\) and \(\bar{X}^{\bu_1}(\cdot) =\beta_1(\cdot)\), where \(\bm{\eta}_1(\cdot)=\{\alpha_1(\cdot),\beta_1(\cdot)\}\) are deterministic functions. This approach is motivated by the fact that, 
for fixed \(x\in\RR^n\) and \(u_2\in\cU_2\),  
\begin{equation*}
	\inf_{u_1 \in \mathcal{U}_1 } J_1( u_1(\cdot),u_2(\cdot)) = \inf_{(\alpha_1, \beta_1) \in (\mathbb{L}^2)^2} \inf_{u_1 \in \mathcal{U}_1 } \left\{ J_1( u_1(\cdot),u_2(\cdot)) : \bar{u}_1(\cdot) =\alpha_1(\cdot), \bar{X}^{\bu_1}(\cdot)  = \beta_1(\cdot) \right\}\footnotemark.
\end{equation*}\footnotetext{If \(\left\{ J_1(u_1(\cdot)) :\bar{u}_1(\cdot) =\alpha_1(\cdot), \bar{X}^{\bu_1}(\cdot)  = \beta_1(\cdot)  \right\}=\emptyset\), then \(\inf_{u_1\in\cU_1}J_1( u_1(\cdot),u_2(\cdot))=\infty\).}
where \(X^{\bu_1}(\cdot)\) denotes the state trajectory under the controls \(\bu_1=\{u_1,u_2\}\).  

We note that the first infimum is related to a constrained control problem, and the state equation of this problem is as follows: for \(s\in[0,T] \),
\begin{equation}\label{eq:st-cons}
	\left\{\begin{aligned}
		&\drm X(s)=\left[\cA_1 X  +\cA_2  \beta_1 
		+\cB_1  u_1  +\cB_2  u_2+b  \right]\drm s+\left[\cC_1  X  +\cC_2  \beta_1  +\cD_1  u_1  +\cD_2  u_2 +\sigma \right]\drm W(s),\\
		&X(0)=x,
	\end{aligned}\right.
\end{equation}
and the cost functional is
\begin{equation}\label{eq:cons}
	\begin{aligned}
		J^{\bm{\eta}_1}_1( u_1(\cdot),u_2(\cdot))=&\mathbb{E}\Big\{\int_0^T\left[\langle Q_1(s)X(s),X(s)\rangle+\langle\bar{Q}_1(s)
		\beta_1(s),\beta_1(s)\rangle+ \langle R_1(s)u_1(s),u_1(s)\rangle\right.\\
		&\left.\qquad\quad+\langle \bar{R}_1(s)\alpha_1(s),\alpha_1(s)\rangle\right]\drm s +\langle G_1 X(T),X(T)\rangle\Big\}.
	\end{aligned}
\end{equation}

The following lemma establishes the strict convexity of the cost functional $J^{\bm{\eta}_1}_1(u_1(\cdot),u_2(\cdot))$ with respect to the control variable $u_1(\cdot)$. The proof follows directly from Theorem \ref{Th:3.1} and is therefore omitted.

\begin{lemma}\label{lem:sc}
	Let (H1) and (H2) hold. For any fixed initial state $x\in\mathbb{R}^n$, leader's control $u_2(\cdot)\in\mathcal{U}_2[0,T]$, and expectation constraint $\bm{\eta}_1\in(\mathbb{L}^2)^2$, the cost functional $J^{\bm{\eta}_1}_1(u_1(\cdot),u_2(\cdot))$ is strictly convex with respect to $u_1(\cdot)$.
\end{lemma}

Building upon this strict convexity, we can directly establish the unique solvability of the constrained optimal control problem, as presented in the following lemma. Since the algebraic arguments closely parallel the affine subspace and coercivity techniques utilized in \cite{xiong2025mean}, the detailed proof is omitted here.

\begin{lemma}\label{lem:F1}
	For any fixed $\bm{\eta}_1 = (\alpha_1, \beta_1) \in (\mathbb{L}^2)^2$, there exists a unique $\tilde{u}_1^{\bm{\eta}_1}(\cdot) \in \mathcal{U}_1[0,T]$ satisfying the expectation constraints $\mathbb{E}[\tilde{u}_1^{\bm{\eta}_1}(\cdot)] = \alpha_1(\cdot)$ and $\mathbb{E}[X^{\tilde{u}_1^{\bm{\eta}_1}}(\cdot)] = \beta_1(\cdot)$, such that
	\[
	J_1^{\bm{\eta}_1}\big( \tilde{u}_1^{\bm{\eta}_1}(\cdot), u_2(\cdot) \big) = \inf_{u_1 \in \mathcal{U}_1} \big\{ J_1^{\bm{\eta}_1}(u_1(\cdot), u_2(\cdot)) \mid \mathbb{E}[u_1(\cdot)] = \alpha_1(\cdot) \text{ and } \mathbb{E}[X^{u_1}(\cdot)] = \beta_1(\cdot) \big\}.
	\]
	Here, $X^{u_1}(\cdot)$ denotes the state trajectory generated by the control $u_1(\cdot)$ for the fixed leader's control $u_2(\cdot) \in \mathcal{U}_2[0,T]$.
\end{lemma}

Then, we introduce the Extended Lagrange multipliers  \(\bm{\lambda}_1=\{\lambda_1(\cdot),\ \ti{\la}_1(\cdot)\} \in (\mathbb{L}^2)^2\) to relax the constraint cost functional given in \eqref{eq:cons}. In this case, the cost functional is as follows  
\begin{equation}\label{eq:cos-ELM1}
	\begin{aligned}	\hat{J}^{\bm{\eta}_1}_1( u_1(\cdot),u_2(\cdot),\bm{\lambda}_1(\cdot))\triangleq & J^{\bm{\eta}_1}_1( u_1(\cdot),u_2(\cdot))
		+2\langle \lambda_1,\mathbb{E}u_1-\alpha_1\rangle_{\mathbb{L}^2}
		+2\langle \ti{\la}_1, \mathbb{E}X^{\bu_1}-\beta_1\rangle_{\mathbb{L}^2}
	\end{aligned}
\end{equation} 
We note that the strict convexity of the cost functional \(\hat{J}^{\bm{\eta}_1}_1( u_1(\cdot),u_2(\cdot),\bm{\lambda}_1(\cdot))\) with respect to the control variable \(u_1(\cdot)\) follows directly  from Lemma \ref{lem:sc}. Furthermore, the concavity of \(\hat{J}^{\bm{\eta}_1}_1( u_1(\cdot),u_2(\cdot),\bm{\lambda}_1(\cdot))\)  with respect to \(\lambda_1\) and \(\ti{\la}_1\) can be established directly by using standard convex analysis arguments. Given the convexity of the sets   \(\cU_1\) and \(\LL^2\) , coupled with the differentiability of \(\hat{J}^{\bm{\eta}_1}_1( u_1(\cdot),u_2(\cdot),\bm{\lambda}_1(\cdot))\) with respect to \(u_1\), \(\lambda_1\), and \(\ti{\la}_1\) respectively, we invoke  Propositions 2.156 and 2.157 from \cite{bonnans2013perturbation} to establish the following min-max duality equality
\[\sup_{(\lambda_1,\ti{\la}_1)\in(\LL^2)^2}\inf_{u_1(\cdot)\in\cU_1} \hat{J}^{\bm{\eta}_1}_1( u_1(\cdot),u_2(\cdot),\bm{\lambda}_1(\cdot))= \inf_{u_1(\cdot)\in\cU_1} \sup_{(\lambda_1,\ti{\la}_1)\in(\LL^2)^2}\hat{J}^{\bm{\eta}_1}_1( u_1(\cdot),u_2(\cdot),\bm{\lambda}_1(\cdot)).\]
Consequently, it is not necessary to prescribe an order for optimizing these variables; the min--max equality ensures that the choice of $u_1$ and $(\lambda_1, \tilde{\lambda}_1)$ is interchangeable.

Therefore, we proceed to solving a control problem with respect to \(u_1(\cdot)\), where \(\bea_1\in(\LL^2)^2\) and \(\bl_1\in(\LL^2)^2\)  are fixed. Accordingly, we formulate this problem as the Problem (F-1) in terms of the Fr\'{e}chet derivative.

\noindent\textbf{Problem (F-1):} For any \(x\in\RR^n\), \(u_2\in\cU_2\), \(\bea_1\in(\LL^2)^2\), and \(\bl_1\in(\LL^2)^2\) fixed. Find a control $\tu^{\bm{\eta}_1,\bm{\lambda}_1}(\cdot) \in \mathcal{U}_1 $ 
such that
\begin{equation*}
	\begin{aligned}
		&D_{u_1} \hat{J}^{\bm{\eta}_1}_1( \tu^{\bm{\eta}_1,\bm{\lambda}_1}(\cdot),u_2(\cdot),\bm{\lambda}_1(\cdot) )=0,
	\end{aligned}
\end{equation*}
Here, \(D_{u_1}\hat{J}^{\bm{\eta}_1}_1( \tu^{\bm{\eta}_1,\bm{\lambda}_1}(\cdot),u_2(\cdot),\bm{\lambda}^{\bm{\eta}_1}_1(\cdot))\) denotes the partial derivative of \(\hat{J}_1^{\bm{\eta}_1}(\cdot)\) with respect to \(u_1(\cdot)\), i.e. for any \(v(\cdot)\in\cU_1\),
\begin{equation}\label{eq:D}
	\begin{aligned}
		& < D_{u_1}\hat{J}_1^{\bm{\eta}_1}( u_1(\cdot),u_2(\cdot),\bm{\lambda_1}^{\bm{\eta}_1}(\cdot)),v(\cdot)>_{\cU_1}\\
		&\equiv\lim_{\epsilon\to 0^+}\frac{\hat{J}_1^{\bm{\eta}_1}( u_1(\cdot)+\epsilon v(\cdot),u_2(\cdot),\bm{\lambda}^{\bm{\eta}_1}_1(\cdot)))-\hat{J}_1^{\bm{\eta}_1}( u_1(\cdot),u_2(\cdot),\bm{\lambda}^{\bm{\eta}_1}_1(\cdot)))}{\epsilon}
	\end{aligned}
\end{equation}

We now state the following theorem, which establishes the stochastic maximum principle for Problem (F-1). Since the proof follows from  standard variational arguments analogous to those used in Theorem \ref{Th:3.2}, the detailed derivation is omitted.
\begin{theorem}\label{th:3.4}
	Let (H1) and (H2) hold. Then for \(x\in\RR^n\),  \(u_2\in\cU_2 \)  and \(\bm{\eta}_1\), \(\bm{\lambda}_1\in(\LL^2)^2\)  fixed, there exists a unique \(\tilde{\tilde{u}}_1^{\bm{\eta}_1,\bm{\lambda_1}}\in\cU_1\) such that \(D_{u_1}\hat{J}_1^{\bm{\eta}_1}=0\). Moreover, \(\tilde{\tilde{u}}_1^{\bm{\eta}_1,\bm{\lambda}_1}(\cdot)\) is optimal if and only if the solution
	\((X^{\bm{\eta}_1,\bm{\lambda}_1}(\cdot),
	Y^{\bm{\eta}_1,\bm{\lambda}_1}(\cdot),
	Z^{\bm{\eta}_1,\bm{\lambda}_1}(\cdot))\)
	to the following FBSDE: for \(\forall s\in[0,T]\),
	\begin{equation}\label{eq:op}
		\left\{	\begin{aligned}
			\drm X^{\bm{\eta}_1,\bm{\lambda}_1}(s)=&\left[
			\cA_1 X^{\bm{\eta}_1,\bm{\lambda}_1} +\cA_2 \beta_1 +\cB_1 \tilde{\tilde{u}}_{1}^{\bm{\eta}_1,\bm{\lambda}_1} 
			+\cB_2 u_2 +b\right]\drm s\\
			& \quad+[\cC_1 X^{\bm{\eta}_1,\bm{\lambda}_1} +\cC_2 \beta_1
			+\cD_1 \tilde{\tilde{u}}_{1}^{\bm{\eta}_1,\bm{\lambda} }+\cD_2 u_2 +\sigma]\drm W(s),\\		
			\drm Y^{\bm{\eta}_1,\bm{\lambda}_1}(s)=&-\left[\cA_1^{\top} Y^{\bm{\eta}_1,\bm{\lambda}_1} +\cC_1^{\top} Z^{\bm{\eta}_1,\bm{\lambda}_1} 		+Q_1 X^{\bm{\eta}_1,\bm{\lambda}_1} 
			+\ti{\la}_1 \right]\drm s
			+Z^{\bm{\eta}_1,\bm{\lambda}_1} \drm W(s),\\
			X^{\bm{\eta}_1,\bm{\lambda}_1}(0)=&x,\qquad Y^{\bm{\eta}_1,\bm{\lambda}_1}(T)=G_1X^{\bm{\eta}_1,\bm{\lambda}_1}(T).
		\end{aligned}\right.
	\end{equation}	
	satisfies the following stationary condition 
	\begin{align}\label{eq:OC1}
		R_1 \tilde{\tilde{u}}_{1}^{\bm{\eta}_1,\bm{\lambda_1}} +\cB_1^{\top} Y^{\bm{\eta}_1,\bm{\lambda}_1}  +\cD_1^{\top} Z^{\bm{\eta}_1,\bm{\lambda}_1} +\lambda_1 =0.
	\end{align}
	
\end{theorem}

Based on Assumption (H2),  and by substituting the equation of \(\ti{\ti{u}}_1^{\bm{\eta}_1,\bm{\lambda_1}}\) given in the above \eqref{eq:OC1} into FBSDE \eqref{eq:op}, we have, for all \(s\in[0,T]\),
\begin{equation}\label{eq:so-u}
	\left\{	\begin{aligned}
		\drm X^{\bea_1,\bl_1} (s)=&\left[\cA_1X^{\bea_1,\bl_1}
		+\cA_2\beta_1-\cB_1R_1^{-1}(\cB_1^{\top}Y^{\bea_1,\bl_1}
		+\cD_1^{\top}Z^{\bea_1,\bl_1}+\lambda_1)
		+\cB_2u_2 +b\right]\drm s\\
		&+[\cC_1X^{\bea_1,\bl_1}+\cC_2\beta_1-\cD_1R_1^{-1}(\cB_1^{\top}Y^{\bea_1,\bl_1}+\cD_1^{\top}Z^{\bea_1,\bl_1}+\lambda_1)+\cD_2u_2 +\sigma]\drm W (s),\\			
		\drm Y^{\bea_1,\bl_1} (s)=&-[\cA_1^{\top}Y^{\bea_1,\bl_1}	+\cC_1^{\top}Z^{\bea_1,\bl_1}+Q_1X^{\bea_1,\bl_1}+\ti{\la}_1]\drm s
		+Z^{\bea_1,\bl_1}\drm W (s),\\
		X^{\bea_1,\bl_1}(0)=&x,\qquad Y^{\bea_1,\bl_1}(T)=G_1X^{\bea_1,\bl_1}(T).
	\end{aligned}\right.
\end{equation}

Thus, the unique solvability of FBSDE \eqref{eq:so-u} is directly derived from Theorem \ref{th:3.4}.
\begin{lemma}\label{lem:3.6}
	Suppose (H1) and (H2) hold. Then, for any \(x\in\RR^n\), \(u_2\in\cU_2\), and \((\bea_1,\bl_1)\in(\LL^2)^4\), the coupled system \eqref{eq:so-u} has a unique adapted solution \((X^{\bea_1,\bl_1}(\cdot),Y^{\bea_1,\bl_1}(\cdot),Z^{\bea_1,\bl_1}(\cdot))\in (L^{2,c}_\FF(\RR^n))^2\times L^2_\FF(\RR^n)\).
\end{lemma}

Combining Lemma~\ref{lem:3.6} with the optimality expression~\eqref{eq:OC1}, the solution to Problem (F-1) is fully characterized.   We  next formulate Problem (F-2), which seeks to optimize over the Lagrange multipliers $(\lambda_1, \tilde{\lambda}_1)$. With the optimal control $\tu_1^{\bea_1,\bl_1}(\cdot)$ now explicitly determined, the cost functional $\hat{J}$ reduces to a functional depending  on $(\lambda_1, \tilde{\lambda}_1)$ and the state trajectories $(X^{\bm{\eta}_1,\bm{\lambda}_1},Y^{\bm{\eta}_1,\bm{\lambda}_1} , Z^{\bm{\eta}_1,\bm{\lambda}_1})$
\begin{equation}\label{eq:cost2}
	\begin{aligned}	
		&\hat{\hat{J}}^{\bea_1}(\la_1(\cdot),\ti{\la}_1(\cdot))\equiv \hat{J}^{\bm{\eta}_1}(\tu^{\bea_1,\bl_1}_1(\cdot),u_2(\cdot),\bm{\lambda}_1(\cdot))\\
		&=\mathbb{E}\Big\{\int_0^T\left[\langle Q_1 X^{\bm{\eta}_1,\bm{\lambda}_1} ,X^{\bm{\eta}_1,\bm{\lambda}_1} \rangle+\langle\bar{Q}_1 \beta_1 ,\beta_1 \rangle + \langle \cB_1^{\top} Y^{\bm{\eta}_1,\bm{\lambda}_1}  +\cD_1^{\top} Z^{\bm{\eta}_1,\bm{\lambda}_1},R^{-1}_1 \left[\cB_1^{\top} Y^{\bm{\eta}_1,\bm{\lambda}_1}  +\cD_1^{\top} Z^{\bm{\eta}_1,\bm{\lambda}_1}  \right]\rangle\right.\\
		&\left.\qquad\quad+\langle \bar{R}_1 \alpha_1 ,\alpha_1 \rangle-\langle \lambda_1, R^{-1}_1\lambda_1+2\alpha_1\rangle+2\langle \ti{\la}_1,X^{\bm{\eta}_1,\bm{\lambda}_1}-\beta_1\rangle\right]\drm s+\langle G_1X^{\bm{\eta}_1,\bm{\lambda}_1}(T),X^{\bm{\eta}_1,\bm{\lambda}_1}(T)\rangle\Big\}.
	\end{aligned}
\end{equation} 

We then formulate this part as Problem (F-2) below in terms of the Fr\'{e}chet derivative.

\noindent\textbf{Problem (F-2):} For any $x\in\RR^n$,  $u_2\in\cU_2$, and \(\bea_1(\cdot)\in(\LL^2)^2\) fixed,  find optimal ELMs \(\bl_1^\ast=(\lambda_1^\ast,\ti{\la}_1^\ast)\in(\LL^2)^2\) such that
\[ D_{\lambda_1} \hat{\hat{J}}^{\bea_1}(\la_1(\cdot),\ti{\la}_1(\cdot))=0,\qquad  D_{\ti{\la}_1}\hat{\hat{J}}^{\bea_1}(\la_1(\cdot),\ti{\la}_1(\cdot))=0.
\]
Here,  the differential operators
$D_{\lambda_1}\hat{\hat{J}}^{\bea_1}(\la_1(\cdot),\ti{\la}_1(\cdot))=0$, and $ D_{\ti{\la}_1}\hat{\hat{J}}^{\bea_1}(\la_1(\cdot),\ti{\la}_1(\cdot))=0$ are defined analogously to that of the differential operator given in \eqref{eq:D}.

The following lemma shows that the expectation constraints \eqref{eq:abco} are automatically recovered at the optimal pair $(\la_1^*,\ti{\la}_1^*)$.
\begin{lemma}\label{lem:3.7}
	Let (H1) and (H2) hold. For any fixed \(x\in\RR^n\), \(u_2(\cdot)\in\cU_2\), and \(\bea_1(\cdot)=(\alpha_1(\cdot),\beta_1(\cdot))\in(\LL^2)^2\), assume that \(\bl_1^\ast=(\lambda_1^\ast,\ti{\la}_1^\ast)\in(\LL^2)^2\) is the optimal pair of ELMs satisfying \(D_{\lambda_1}\hat{\hat{J}}^{\bea_1}(\la_1(\cdot),\ti{\la}_1(\cdot))=0\) and \(D_{\ti{\la}_1}\hat{\hat{J}}^{\bea_1}(\la_1(\cdot),\ti{\la}_1(\cdot))=0\). Then, the solution triple \((X^{\bea_1,\bl^\ast_1}(\cdot),Y^{\bea_1,\bl^\ast_1}(\cdot),Z^{\bea_1,\bl^\ast_1}(\cdot))\) to the FBSDE \eqref{eq:so-u} with \(\lambda_1\) and \(\ti{\la}_1\) replaced by \(\lambda_1^\ast\) and \(\ti{\la}_1^\ast\) respectively, satisfies the following conditions:
	\begin{equation}\label{eq:abco}
		\EE \tu_1^{\bea_1,\bl_1^\ast}(\cdot)=\alpha_1(\cdot),\qquad \EE X^{\bea_1,\bl^\ast_1}(\cdot)=\beta_1(\cdot).  
	\end{equation}
\end{lemma}

We proceed to discussing the specific form of the optimal pair \((\lambda_1^\ast(\cdot), \tilde{\lambda}_1^\ast(\cdot))\).   Based on the unique solvability of the linear FBSDE \eqref{eq:so-u}, we can define linear operators \(\cP_{i,j}\) (where \(i=1,2,3\) and \(j=1,\cdots,6\)) satisfying that for \(i=1,2\), \(\cP_{i,1}: \mathbb{R}^n \to L^{2,c}_{\FF}(\mathbb{R}^n)\), \(\cP_{i,m}: L^2 \to L^{2,c}_{\FF}(\mathbb{R}^n)\) (for \(m=2,\cdots,5\)), and \(\cP_{i,6}: \cU_2 \to L^{2,c}_{\FF}(\mathbb{R}^n)\); for \(i=3\), \(\cP_{3,1}: \mathbb{R}^n \to L^2_{\FF}(\mathbb{R}^n)\), \(\cP_{3,m}: L^2 \to L^2_{\FF}(\mathbb{R}^n)\) (for \(m=2,\cdots,5\)), and \(\cP_{3,6}: \cU_2 \to L^2_{\FF}(\mathbb{R}^n)\); along with some random variables \(\cP_{1,7}, \cP_{2,7} \in L^{2,c}_{\FF}(\mathbb{R}^n)\) and \(\cP_{3,7} \in L^2_{\FF}(\mathbb{R}^n)\), such that
\begin{equation*}
	\begin{aligned}
		X^{\bea_1,\bl_1}(\cdot)=&(\cP_{1,1}x)(\cdot)+(\cP_{1,2}\lambda^\ast_1)(\cdot)+(\cP_{1,3}\lambda^\ast_2)(\cdot)+(\cP_{1,4}\alpha_1)(\cdot)+(\cP_{1,5}\beta_1)(\cdot)+(\cP_{1,6}u_2)(\cdot)+\cP_{1,7}(\cdot),\\
		Y^{\bea_1,\bl_1}(\cdot)=&(\cP_{2,1}x)(\cdot)+(\cP_{2,2}\lambda^\ast_1)(\cdot)+(\cP_{2,3}\lambda^\ast_2)(\cdot)+(\cP_{2,4}\alpha_1)(\cdot)+(\cP_{2,5}\beta_1)(\cdot)+(\cP_{2,6}u_2)(\cdot)+\cP_{2,7}(\cdot),\\
		Z^{\bea_1,\bl_1}(\cdot)=&(\cP_{3,1}x)(\cdot)+(\cP_{3,2}\lambda^\ast_1)(\cdot)+(\cP_{3,3}\lambda^\ast_2)(\cdot)+(\cP_{3,4}\alpha_1)(\cdot)+(\cP_{3,5}\beta_1)(\cdot)+(\cP_{3,6}u_2)(\cdot)+\cP_{3,7}(\cdot).
	\end{aligned}
\end{equation*}
Based on the expression of \(\tu_1^{\bea_1,\bl_1}(\cdot)\), we define linear operators \(\cP_{i,j}\) (where \(i=4\) and \(j=1,2,\cdots,6\)), with   \(\cP_{4,1}:\mathbb{R}^n \to \cU_1\), \(\cP_{4,m}:\LL^2 \to \cU_1\) (for \(m=2,3,\cdots,5\)), \(\cP_{4,6}:\cU_2 \to \cU_1\), and the random variable $\cP_{4,7}\in\cU_1$,  such that
\begin{equation}\label{eq:ou}
	\begin{aligned}
		\tu_1^{\bea_1,\bl_1}(\cdot)&=(\cP_{4,1}x)(\cdot)+(\cP_{4,2}\lambda^\ast_1)(\cdot)+(\cP_{4,3}\lambda^\ast_2)(\cdot)+(\cP_{4,4}\alpha_1)(\cdot)+(\cP_{4,5}\beta_1)(\cdot)+(\cP_{4,6}u_2)(\cdot)+\cP_{4,7}(\cdot).
	\end{aligned}
\end{equation}
Moreover, from the conditions \eqref{eq:abco} and the fact that expectation is also a linear operator, we can also define some linear operators \(\cO_{i,j}\)   with \(\cO_{i,1}:\RR^n\to \LL^2\), \(\cO_{i,m}:\LL^2\to \LL^2\), \(\cO_{i,6}:\cU_2\to \LL^2\), and \(\cO_{i,7}\in \LL^2\), where \(i=1,2\), \(j=1,\cdots,6\), and \(m=2,\cdots,5\), such that
\begin{equation*}
	\begin{aligned}
		\alpha_1(\cdot)=&(\cO_{1,1}x)(\cdot)+(\cO_{1,2}\lambda^\ast_1)(\cdot)+(\cO_{1,3}\lambda^\ast_2)(\cdot)+(\cO_{1,4}\alpha_1)(\cdot)+(\cO_{1,5}\beta_1)(\cdot)+(\cO_{1,6}u_2)(\cdot)+\cO_{1,7}(\cdot),\\
		\beta_1(\cdot)=&(\cO_{2,1}x)(\cdot)+(\cO_{2,2}\lambda^\ast_1)(\cdot)+(\cO_{2,3}\lambda^\ast_2)(\cdot)+(\cO_{2,4}\alpha_1)(\cdot)+(\cO_{2,5}\beta_1)(\cdot)+(\cO_{2,6}u_2)(\cdot)+\cO_{2,7}(\cdot).
	\end{aligned}
\end{equation*}
Therefore,  the aforementioned system of equations in matrix form can be rewritten  as follows
\begin{align}\label{eq:laeq}
	\tilde{\cO}_1(\bl_1^\ast)^\top=\tilde{\cO}_2\bea_1^\top-\tilde{\cO}_3x-\tilde{\cO}_4 u_2-\ti{\cO}_5,
\end{align}
where \(\tilde{\cO}_1=\begin{pmatrix}
	\cO_{1,2}&\cO_{1,3}\\
	\cO_{2,2}&\cO_{2,3}
\end{pmatrix}\), \(\tilde{\cO}_2=I_{2\times 2}-\begin{pmatrix}
	\cO_{1,4}&\cO_{1,5}\\
	\cO_{2,4}&\cO_{2,5}
\end{pmatrix}\), \(\tilde{\cO}_3=(\cO_{1,1},\cO_{2,1})^\ast\),   \(\tilde{\cO}_4=(\cO_{1,6},\cO_{2,6})^\ast\), and \(\ti{\cO}_5=(\cO_{1,7},\cO_{2,7})^\ast\). We then derive that \(\tilde{\cO}_i\) (for \(i=1,2,3,4\)) are  operator-valued matrices, where \(\tilde{\cO}_1:\LL^2\to \LL^2\), \(\tilde{\cO}_2:\LL^2\to \LL^2\), \(\tilde{\cO}_3:\RR^n\to \LL^2\),   \(\tilde{\cO}_4:\cU_2\to \LL^2\), while  \(\ti{\cO}_5\in\LL^2\).

Next, we consider two cases for analyzing the solvability of the equation \eqref{eq:laeq}.

\textbf{Case 1:}If the matrix \(\tilde{\cO}_1\) is  invertible, then equation \eqref{eq:laeq} is uniquely solvable. This implies that \((\bl_1^\ast)^\top = \tilde{\cO}_1^{-1}\left(\tilde{\cO}_2 \bea_1^\top - \tilde{\cO}_3x-\tilde{\cO}_4 u_2-\ti{\cO}_5\right)\). Then, both $\lambda_1^\ast$ and $\tilde{\lambda}_1^\ast$ can be formulated as linear combinations of $ \alpha_1$, $\beta_1$, $x$, $u_2$, together with a homogeneous  term.  Substituting this expression into the representation of $\tilde{\ti{u}}_1^{\bm{\bea}_1,\bm{\la}_1}$ given in \eqref{eq:ou} shows that  $\tilde{\ti{u}}_1^{\bm{\bea}_1,\bm{\la}_1}$ can be rewritten by some affine operators acting on the tuple $( \alpha_1,\beta_1, x, u_2)$.

\textbf{Case 2:} If the matrix \(\tilde{\cO}_1\) is not invertible,  the  equation \eqref{eq:laeq} is ill-posed in the sense of unique solvability. To derive the operator representation of \(\tu_1^{\bea_1,\bl_1}(\cdot)\), we have to  characterize the dual solution space and its projection onto the primal control space \(\cU_1\).

Let $\mathbf{r} \equiv \tilde{\cO}_2\bea_1^\top 
- \tilde{\cO}_3 x - \tilde{\cO}_4 u_2 - \tilde{\cO}_5$.
The existence of the optimal ELM pair 
$\bl_1^\ast = (\lambda_1^\ast, \tilde\lambda_1^\ast)$ 
follows from the first-order conditions 
$D_{\lambda_1}\hat{\hat{J}}^{\bea_1} = 0$ and 
$D_{\tilde\lambda_1}\hat{\hat{J}}^{\bea_1} = 0$ 
of Problem~(F-2), which are well-defined by the 
differentiability of $\hat{\hat{J}}^{\bea_1}$ 
with respect to $\bl_1$. By Lemma~\ref{lem:3.7}, 
this $\bl_1^\ast$ satisfies the consistency 
conditions~\eqref{eq:abco}. This guarantees that 
the Fréchet derivative of the constraint operator 
in Problem~(F-2) is surjective, i.e., the regular 
point condition of the Generalized Lagrange 
Multiplier Theorem 
\cite[Theorem~1, p.243]{luenberger1997optimization} 
is satisfied. Hence $\bl_1^\ast$ is a solution 
to~\eqref{eq:laeq}, which establishes 
$\mathbf{r} \in \mathrm{Ran}(\tilde{\cO}_1)$. The complete set of valid extended Lagrange multipliers forms a non-empty affine subspace 
\begin{equation}
	\mathscr{S}  := \left\{(\bm{\lambda}_1^\ast)^\top + \bm{\delta} \mid\bm{\delta} \in \ker(\tilde{\mathcal{O}}_1)\right\} \subset (\mathbb{L}^2)^2 
\end{equation}
where \(\bl_1^\ast\) is any particular solution to \eqref{eq:laeq}. 

To bridge the non-unique dual space \(\mathscr{S}\) with the primal control, we define the combined block operator \(\bm{\mathcal{P}}_{4,\lambda} \equiv (\cP_{4,2}, \cP_{4,3}): (\LL^2)^2 \to \cU_1\). Since the cost functional $\hat{J}_1^{\bm{\eta}_1}$ is strictly  convex in $u_1$,  its minimiser over $\mathcal{U}_1$ is 
unique. Therefore, the mapping from the solution space $\mathscr{S}$ to the control space $\mathcal{U}_1$ is constant. For any two solutions $\bm{\lambda}_1^{(1)}, 
\bm{\lambda}_1^{(2)} \in \mathscr{S}$, their 
difference $\bm{\delta} \equiv \bm{\lambda}_1^{(1)} 
- \bm{\lambda}_1^{(2)}$ lies in $\ker(\tilde{\cO}_1)$. 
Since both $\bm{\lambda}_1^{(1)}$ and 
$\bm{\lambda}_1^{(2)}$ satisfy the consistency 
conditions~\eqref{eq:abco}, the strict convexity of 
$\hat{J}_1^{\bm{\eta}_1}$ in $u_1$  from the Theorem~\ref{Th:3.1}  implies that both yield the 
same optimal control. Therefore 
$\bm{\cP}_{4,\lambda}(\bm{\delta}) = 0$, which gives 
$\ker(\tilde{\cO}_1) \subseteq 
\ker(\bm{\cP}_{4,\lambda})$.

This inclusion ensures that \(\bm{\mathcal{P}}_{4,\lambda}\) factors well-defined through the quotient space \((\LL^2)^2 / \ker(\tilde{\cO}_1)\). In other words, any multiplier perturbation \(\bm{\delta} \in \ker(\tilde{\cO}_1)\) is inherently absorbed and maps to the zero element in \(\cU_1\). Thus, by fixing an arbitrary particular solution \(\bl_1^\ast \in \mathscr{S}\), the optimal control precisely preserves the identical affine representation as in Case 1.

Consequently, there exist linear operators $\mathcal{L}_{i,j}$ and random variables $\mathcal{L}_{i,5} \in \mathbb{L}^2$ (for $i=1,2$ and $j=1,2,3,4$), where $\mathcal{L}_{i,m}: \mathbb{L}^2 \to \mathbb{L}^2$ (for $m=1,2$), $\mathcal{L}_{i,3}: \mathbb{R}^n \to \mathbb{L}^2$, and $\mathcal{L}_{i,4}: \mathcal{U}_2 \to \mathbb{L}^2$, such that
\begin{align*}
	\begin{pmatrix}
		\la_1^\ast\\
		\ti{\la}_1^\ast
	\end{pmatrix}=\begin{pmatrix}
		\cL_{1,1}&\cL_{1,2}\\
		\cL_{2,1}&\cL_{2,2}
	\end{pmatrix}\begin{pmatrix}
		\alpha_1\\
		\beta_1
	\end{pmatrix}+\begin{pmatrix}
		\cL_{1,3}\\
		\cL_{2,3}
	\end{pmatrix} x+\begin{pmatrix}
		\cL_{1,4}\\
		\cL_{2,4}
	\end{pmatrix} u_2+\begin{pmatrix}
		\cL_{1,5}\\
		\cL_{2,5}
	\end{pmatrix} .
\end{align*}
This allows us to further deduce the existence of linear operators  \(\cK_{1,1}:\RR^n\to \cU_1\), \(\cK_{1,j}:\LL^2\to \cU_1\)(for \(j=2,3\)),   \(\cK_{1,4}:\cU_2\to \cU_1\), and \(\cK_{1,5}\in\cU_1\) such that 
\begin{align}\label{eq:ou1}
	\tu_1^{\bea_1,\bl_1}(\cdot)=(\cK_{1,1} x)(\cdot)+(\cK_{1,2} \alpha_1)(\cdot)+(\cK_{1,3}\beta_1)(\cdot)+(\cK_{1,4} u_2)(\cdot)+\cK_{1,5}(\cdot).
\end{align}

Finally, we formulate the optimal control problem with respect to the new control variables $\alpha_1(\cdot)$ and $\beta_1(\cdot)$,  which will be formulated as Problem (F-3) below. Its state equation is \eqref{eq:st-cons}, with \(u_1\) replaced by \(\tu_1^{\bea_1,\bl_1}\) and the state variable \(X\) replaced by \(X^{\bea_1,\bl_1}\), and the corresponding cost functional is
\begin{equation}\label{eq:cost3}
	\begin{aligned}  &\tilde{J}_1(\alpha_1(\cdot),\beta_1(\cdot))\\
		&\equiv\mathbb{E}\Big\{\int_0^T\Big[\langle Q_1(s)X^{\bea_1,\bl_1}(s),X^{\bea_1,\bl_1}(s)\rangle+\langle\bar{Q}_1(s)
		\beta_1(s),\beta_1(s)\rangle+ \langle R_1(s)\tu_1^{\bea_1,\bl_1}(s),\tu_1^{\bea_1,\bl_1}(s)\rangle\\
		&\qquad\qquad+\langle \bar{R}_1(s)\alpha_1(s),\alpha_1(s)\rangle\Big]\drm s +\langle G_1X^{\bea_1,\bl_1}(T),X^{\bea_1,\bl_1}(T)\rangle \Big\}.
	\end{aligned}
\end{equation}
Moreover, Problem (F-3) is formulated in terms of the following Fréchet derivative.

\noindent\textbf{Problem (F-3):} For any fixed $x\in\RR^n$ and $u_2\in\cU_2$, find optimal control variables $\alpha_1^\ast(\cdot)$ and $\beta_1^\ast(\cdot)$ such that
\[
D_{\alpha_1}\tilde{J}_1(\alpha_1^\ast(\cdot),\beta_1^\ast(\cdot))=0,\qquad  D_{\beta_1}\tilde{J}_1(\alpha_1^\ast(\cdot),\beta_1^\ast(\cdot))=0,
\]
where the differential operators are defined identically to \eqref{eq:D}.

By the linearity of the SDE \eqref{eq:st-cons} and the affine representation of $\tu_1^{\bea_1,\bl_1}$ in \eqref{eq:ou1}, the state process $X^{\bea_1,\bl_1}(\cdot)$ and its terminal value $X^{\bea_1,\bl_1}(T)$ inherently admit affine representations. Consequently, taking expectations and applying $\EE X^{\bea_1,\bl_1}(T)=\beta_1(T)$, there exist bounded linear operators $\cK_{i,j}$ and elements $\cK_{i,5}$ (for $i=2,3,4$ and $j=1,2,3,4$) with the following specific mappings
\begin{itemize}
	\item$\cK_{2,1}: \RR^n \to L^{2,c}_\FF(\RR^n)$, $\cK_{2,m}: \LL^2 \to L^{2,c}_\FF(\RR^n)$ (for $m=2,3$), $\cK_{2,4}: \cU_2 \to L^{2,c}_\FF(\RR^n)$, and $\cK_{2,5} \in L^{2,c}_\FF(\RR^n)$;
	\item$\cK_{3,1}: \RR^n \to L^2_{\cF_T}(\RR^n)$, $\cK_{3,m}: \LL^2 \to L^2_{\cF_T}(\RR^n)$ (for $m=2,3$), $\cK_{3,4}: \cU_2 \to L^2_{\cF_T}(\RR^n)$, and $\cK_{3,5} \in L^2_{\cF_T}(\RR^n)$;
	\item$\cK_{4,1}: \RR^n \to \LL^2$, $\cK_{4,m}: \LL^2 \to \LL^2$ (for $m=2,3$), $\cK_{4,4}: \cU_2 \to \LL^2$, and $\cK_{4,5} \in \LL^2$,
\end{itemize}
such that 
\begin{equation}\label{eq:bo}
\begin{aligned}
	X^{\bea_1,\bl_1}(\cdot) &= (\cK_{2,1}x)(\cdot) + (\cK_{2,2}\alpha_1)(\cdot) + (\cK_{2,3}\beta_1)(\cdot) + (\cK_{2,4}u_2)(\cdot) + \cK_{2,5}(\cdot), \\
	X^{\bea_1,\bl_1}(T) &= \cK_{3,1}x + \cK_{3,2}\alpha_1 + \cK_{3,3}\beta_1 + \cK_{3,4}u_2 + \cK_{3,5}, \\
	\beta_1(T) &= \cK_{4,1}x + \cK_{4,2}\alpha_1 + \cK_{4,3}\beta_1 + \cK_{4,4}u_2 + \cK_{4,5}.
\end{aligned}
\end{equation}

Now, we introduce the following lemma to state the strict convexity of \(\tilde{J}_1(\alpha_1(\cdot),\beta_1(\cdot))\) with respect to the control variables \(\alpha_1(\cdot)\) and \(\beta_1(\cdot)\), which implies the uniqueness of the optimal control variables \(\alpha^\ast_1(\cdot)\) and \(\beta_1^\ast(\cdot)\).
\begin{lemma}\label{lem:cc}
	Let (H1) and (H2) hold. Then, for any \(x\in\RR^n\) and \(u_2\in\cU_2\), the cost functional \(\tilde{J}_1(\alpha_1(\cdot),\beta_1(\cdot))\) is strictly convex with respect to \(\alpha_1\) and \(\beta_1\).
\end{lemma}

Based on this convexity, we present the necessary and sufficient conditions for the optimal pair $(\alpha_1^\ast,\beta_1^\ast)$.

\begin{theorem}\label{th:3.10}
	Let assumptions (H1) and (H2) hold. For any fixed $x\in\RR^n$ and $u_2\in\cU_2$, $(\alpha_1^\ast,\beta_1^\ast)\in(\LL^2)^2$ is the optimal pair if and only if the following operator equation admits a unique solution:
	\begin{equation}\label{eq:eo}
		(\cW+\cK_{23}^\ast\cT\cK_{23})\cdot(\alpha^\ast_1,\beta_1^\ast)^\top+\cK_{23}^\ast\cT \cK_{14}\cdot (x,u_2)^\top+\cK_{23}^\ast\cT\cK_5=(0,0)^\top,
	\end{equation}
	where $\cW=\begin{pmatrix} \bar{R}_1&0\\ 0&\bar{Q}_1 \end{pmatrix}$, $\cK_{23}=\begin{pmatrix} \cK_{1,2}&\cK_{1,3}\\ \cK_{2,2}&\cK_{2,3}\\ \cK_{3,2}&\cK_{3,3} \end{pmatrix}$, $\cT=\begin{pmatrix} R_1&0&0\\ 0&Q_1&0\\ 0&0&G_1 \end{pmatrix}$, $\cK_{14}=\begin{pmatrix} \cK_{1,1}&\cK_{1,4}\\ \cK_{2,1}&\cK_{2,4}\\ \cK_{3,1}&\cK_{3,4} \end{pmatrix}$, and $\cK_5=(\cK_{1,5},\cK_{2,5},\cK_{3,5})^\ast$.
\end{theorem}

\begin{remark}
	By Assumption (H2), equation \eqref{eq:eo} admits a unique solution pair $(\alpha_1^\ast,\beta_1^\ast)$ which is affine in $x$ and $u_2$. Substituting this solution back into \eqref{eq:ou1} yields the explicit affine feedback form for the optimal control. Specifically, there exist linear operators $\cM_{1,1}: \RR^n\to \cU_1$, $\cM_{1,2}:\cU_2\to \cU_1$, and a random variable $\cM_{1,3}\in\cU_1$ such that:
	\begin{equation}\label{eq:3u}
		\tu_1^{\bea_1,\bl_1}(\cdot)=(\cM_{1,1}x)(\cdot)+(\cM_{1,2}u_2)(\cdot)+\cM_{1,3}(\cdot).
	\end{equation}
\end{remark}

We summarize the complete solution to Problem (MFSOLQ-F) in the following main theorem.

\begin{theorem}\label{th:follow-main}
	Let Assumptions (H1) and (H2) hold. The unique optimal control $\tilde{u}_1(\cdot)$ of Problem (MFSOLQ-F) is given by \eqref{eq:OC1}, where the state processes $(X^{\bea_1,\bl_1}, Y^{\bea_1,\bl_1}, Z^{\bea_1,\bl_1})$ solve the FBSDE \eqref{eq:op} parameterized by the optimal extended Lagrange multipliers $(\lambda_1^\ast, \ti{\la}_1^\ast)$ and $\beta_1^\ast$. Furthermore, $(\lambda_1^\ast, \ti{\la}_1^\ast)$ are determined by \eqref{eq:laeq}, with the optimal pair $(\alpha_1^\ast, \beta_1^\ast)$ being the unique solution to \eqref{eq:eo}.
\end{theorem}

\begin{remark}
From a computational perspective, the theoretical characterization in Theorem \ref{th:follow-main} provides a natural foundation for developing an iterative numerical scheme. While a rigorous contraction-based justification of the Picard iteration for the fully coupled FBSDE \eqref{eq:op} typically necessitates an additional small-horizon condition, preserving the inherent forward-backward coupling during the iterative process proves to be highly effective in practice. Therefore, rather than employing artificial decoupling strategies, we directly tackle the coupled leader-follower system. This coupled iterative scheme serves as the core mechanism for the numerical solver developed in our subsequent experiments.
\end{remark}

\subsection{Solving the Problem (MFSOLQ-L)}\label{sec:4}

By Theorem \ref{th:follow-main}, for any $x\in\mathbb{R}^n$ and $u_2(\cdot)\in\mathcal{U}_2[0,T]$, the follower's optimal response admits the affine representation \eqref{eq:3u}, where $\mathcal{M}_{1,1}:\mathbb{R}^n\to\mathcal{U}_1[0,T]$ and $\mathcal{M}_{1,2}:\mathcal{U}_2[0,T]\to\mathcal{U}_1[0,T]$ are bounded linear operators and $\mathcal{M}_{1,3}\in\mathcal{U}_1[0,T]$.

Substituting \eqref{eq:3u} into the state equation \eqref{state1} yields
\begin{equation}\label{state3}
	\left\{
	\begin{aligned}
		\drm X(t) &= \bigl[(\cA_1 X)(t) + (\cA_2 \bar X)(t)
		+ (\tilde{\cB}_2 u_2)(t) + \tilde b(t)\bigr]\drm t\\
		&\quad + \bigl[(\cC_1 X)(t) + (\cC_2 \bar X)(t)
		+ (\tilde{\cD}_2 u_2)(t) + \tilde\sigma(t)\bigr]\drm W(t),\\
		X(0) &= x,
	\end{aligned}\right.
\end{equation}
where
\[
\tilde{\cB}_2 = \cB_1 \mathcal{M}_{1,2} + \cB_2, \qquad \tilde{\cD}_2 = \cD_1 \mathcal{M}_{1,2} + \cD_2,
\]
\[
\tilde b = \cB_1 \mathcal{M}_{1,1} x + \cB_1 \mathcal{M}_{1,3} + b, \qquad \tilde\sigma = \cD_1 \mathcal{M}_{1,1} x + \cD_1 \mathcal{M}_{1,3} + \sigma.
\]

Since $\mathcal{M}_{1,2}:\mathcal{U}_2[0,T]\to\mathcal{U}_1[0,T]$ is bounded and $\cB_1,\cB_2,\cD_1,\cD_2$ satisfy (H1), the aggregated coefficients $\tilde{\cB}_2$ and $\tilde{\cD}_2$ remain in $\mathcal{L}^\infty_\mathbb{F}(L^2_{\mathcal{F}_T}(\mathbb{R}^{m_2});L^2_{\mathcal{F}_T}(\mathbb{R}^n))$. Likewise, the inhomogeneous terms satisfy $\tilde b(\cdot),\tilde\sigma(\cdot)\in L^2_\mathbb{F}(\mathbb{R}^n)$. The leader's cost functional is given by \eqref{cost0} with the index $i=2$. 

Thus, Problem (MFSOLQ-L) can be solved using the same approach applied to Problem (MFSOLQ-F). Specifically, by introducing the corresponding extended Lagrange multipliers $\bm{\lambda}_2 = (\lambda_2, \tilde{\lambda}_2)$ and imposing the analogous expectation constraints $\bm{\eta}_2 = (\alpha_2, \beta_2)$ for the leader, we can fully characterize the leader's optimal control.  

To avoid redundancy, we omit the detailed derivations and directly present the complete characterization of the leader's problem in the following main theorem.

\begin{theorem}\label{th:leader_main}
	Let Assumptions (H1) and (H2) hold. For any fixed $x \in \mathbb{R}^n$, the unique optimal control $\tilde{u}_2^\ast(\cdot)$ for Problem (MFSOLQ-L) is given by:
	\begin{equation}\label{eq:OC2}
		\tilde{u}_2^\ast(\cdot) = -R^{-1}_2 \left[\tilde{\cB}_2^{\top} Y^{\bm{\eta}_2^\ast,\bm{\lambda}_2^\ast} + \tilde{\cD}_2^{\top} Z^{\bm{\eta}_2^\ast,\bm{\lambda}_2^\ast} + \lambda_2^\ast \right],
	\end{equation}
	where the associated optimal state processes $(X^{\bm{\eta}_2^\ast,\bm{\lambda}_2^\ast}(\cdot), Y^{\bm{\eta}_2^\ast,\bm{\lambda}_2^\ast}(\cdot), Z^{\bm{\eta}_2^\ast,\bm{\lambda}_2^\ast}(\cdot))$ solve the following coupled linear FBSDE:
	\begin{equation}\label{eq:4so-u}
		\left\{	
		\begin{aligned}
			\drm X^{\bm{\eta}_2^\ast,\bm{\lambda}_2^\ast}(s) &= \Big[\cA_1X^{\bm{\eta}_2^\ast,\bm{\lambda}_2^\ast} + \cA_2\beta_2^\ast - \tilde{\cB}_2R_2^{-1}(\tilde{\cB}_2^{\top}Y^{\bm{\eta}_2^\ast,\bm{\lambda}_2^\ast} + \tilde{\cD}_2^{\top}Z^{\bm{\eta}_2^\ast,\bm{\lambda}_2^\ast} + \lambda_2^\ast) + \tilde{b}\Big]\drm s \\
			&\quad + \Big[\cC_1X^{\bm{\eta}_2^\ast,\bm{\lambda}_2^\ast} + \cC_2\beta_2^\ast - \tilde{\cD}_2R_2^{-1}(\tilde{\cB}_2^{\top}Y^{\bm{\eta}_2^\ast,\bm{\lambda}_2^\ast} + \tilde{\cD}_2^{\top}Z^{\bm{\eta}_2^\ast,\bm{\lambda}_2^\ast} + \lambda_2^\ast) + \tilde{\sigma}\Big]\drm W(s), \\			
			\drm Y^{\bm{\eta}_2^\ast,\bm{\lambda}_2^\ast}(s) &= -\Big[\cA_1^{\top}Y^{\bm{\eta}_2^\ast,\bm{\lambda}_2^\ast} + \cC_1^{\top}Z^{\bm{\eta}_2^\ast,\bm{\lambda}_2^\ast} + Q_2X^{\bm{\eta}_2^\ast,\bm{\lambda}_2^\ast} + \tilde{\lambda}_2^\ast\Big]\drm s + Z^{\bm{\eta}_2^\ast,\bm{\lambda}_2^\ast}\drm W(s), \\
			X^{\bm{\eta}_2^\ast,\bm{\lambda}_2^\ast}(0) &= x, \qquad Y^{\bm{\eta}_2^\ast,\bm{\lambda}_2^\ast}(T) = G_2X^{\bm{\eta}_2^\ast,\bm{\lambda}_2^\ast}(T).
		\end{aligned}
		\right.
	\end{equation}
	
	Furthermore, the optimal extended Lagrange multipliers $\bm{\lambda}_2^\ast = (\lambda_2^\ast, \tilde{\lambda}_2^\ast)$ are explicitly determined by the expectation constraints $\bm{\eta}_2^\ast = (\alpha_2^\ast, \beta_2^\ast)$ through the analogous affine representation derived for the follower's problem. This optimal pair $(\alpha_2^\ast, \beta_2^\ast) \in (\mathbb{L}^2)^2$ is the unique solution to the following operator equation:
	\begin{equation}\label{eq:4eo} 
		(\tilde{\mathcal{W}} + \tilde{\mathcal{K}}_{23}^\ast\tilde{\mathcal{T}}\tilde{\mathcal{K}}_{23})\cdot(\alpha^\ast_2,\beta_2^\ast)^\top + \tilde{\mathcal{K}}_{23}^\ast\tilde{\mathcal{T}}\tilde{\mathcal{K}}_{14}\cdot x + \tilde{\mathcal{K}}_{23}^\ast\tilde{\mathcal{T}}\tilde{\mathcal{K}}_5 = (0,0)^\top,
	\end{equation}
	where $\tilde{\mathcal{W}}=\begin{pmatrix} \bar{R}_2&0\\ 0&\bar{Q}_2 \end{pmatrix}$, $\tilde{\mathcal{K}}_{23}=\begin{pmatrix} \tilde{\mathcal{K}}_{1,2}&\tilde{\mathcal{K}}_{1,3}\\ \tilde{\mathcal{K}}_{2,2}&\tilde{\mathcal{K}}_{2,3}\\ \tilde{\mathcal{K}}_{3,2}&\tilde{\mathcal{K}}_{3,3} \end{pmatrix}$, $\tilde{\mathcal{T}}=\begin{pmatrix} R_2&0&0\\ 0&Q_2&0\\ 0&0&G_2 \end{pmatrix}$, $\tilde{\mathcal{K}}_{14}=\begin{pmatrix} \tilde{\mathcal{K}}_{1,1}&\tilde{\mathcal{K}}_{1,4}\\ \tilde{\mathcal{K}}_{2,1}&\tilde{\mathcal{K}}_{2,4}\\ \tilde{\mathcal{K}}_{3,1}&\tilde{\mathcal{K}}_{3,4} \end{pmatrix}$, and $\tilde{\mathcal{K}}_5 = (\tilde{\mathcal{K}}_{1,5},\tilde{\mathcal{K}}_{2,5},\tilde{\mathcal{K}}_{3,5})^\ast$.
\end{theorem}

\section{Numerical Implementation and Validation}
\label{sec:5}
In this section, we numerically validate the optimal controls $\tilde{u}_1$ and $\tilde{u}_2$ derived in Sections~\ref{sec:3} and~\ref{sec:4}. For systems with stochastic operator-valued coefficients, the associated stochastic operator-valued Riccati equations preclude the direct application of conventional PDE-based methods. To address this challenge and solve our problem numerically, we develop the Deep FBSDE Picard Solver (DFPS). This framework integrates Picard fixed-point iterations with neural parameterization to resolve the coupled leader–follower FBSDE system, thereby bypassing the explicit construction of stochastic Riccati equations.

The remainder of this section is organized as follows. Subsection~\ref{sub:dis} details the discretization and parameterization of the proposed algorithm. This includes the network architectures, the enforcement of mean-field consistency, the augmented Lagrangian formulation with asymptotic feasibility guarantees, and a summary of the complete numerical procedure. Subsection~\ref{sec:experiments} then provides comprehensive numerical experiments to illustrate the algorithm's performance and robustness. Specifically, these experiments encompass convergence and feasibility diagnostics, discretization sensitivity analysis, a Riccati sanity check, ablation studies, equilibrium validation, and a financial application.

\subsection{Discretization and Parameterization}\label{sub:dis}
To transition from the theoretical framework to numerical simulation, we restrict our state and control variables to finite-dimensional Euclidean spaces. In this concrete setting, the abstract bounded linear operators $\mathcal{A}_i(\cdot), \mathcal{B}_i(\cdot), \mathcal{C}_i(\cdot)$, and $\mathcal{D}_i(\cdot)$ defined in Assumption (H1) are naturally realized as adapted matrix-valued stochastic processes, which we denote by $A_i(\cdot), B_i(\cdot), C_i(\cdot)$, and $D_i(\cdot)$, respectively. Consequently, the abstract operator actions reduce to standard matrix-vector multiplications.

\begin{remark}\label{rem:op_vs_matrix}
The operator-valued formulation used in Section~\ref{sec2} is not 
merely a matter of abstraction. In the Stackelberg setting of this 
paper, solving the follower's problem induces an affine operator 
representation of the optimal response, in which the operators act 
on the initial state, the leader's control, and a non-homogeneous 
term. After substituting this follower response into the state 
equation, the leader faces effective dynamics with random 
operator-valued coefficients, even when the primitive coefficients 
are finite-dimensional stochastic matrices. In the numerical 
implementation, these operators are realized through adapted 
matrix-valued processes and the response sensitivities 
$\mathcal{M}_{1,1,k}$ and $\mathcal{M}_{1,2,k}$ (see~\eqref{eq:3u}), which are extracted from the trained follower 
network via automatic differentiation.
\end{remark}
We discretize the continuous-time model on the uniform grid
\[
\pi_N:\ 0=t_0<t_1<\cdots<t_N=T,
\qquad \Delta t = T/N.
\]
Let $\Delta W_k^{(m)} \sim \mathcal{N}(0,\Delta t\,I_d)$ be independent Brownian increments. Then, for the $m$-th simulated path ($m=1,\dots,M$), the state equation \eqref{state1} is discretised via the Euler--Maruyama scheme as
\begin{equation}\label{eq:euler}
	\begin{aligned}
		X^{(m)}_{k+1}
		=&\, X^{(m)}_k
		+ \bigl[ A_{1,k}^{(m)} X^{(m)}_k + A_{2,k}^{(m)}\bar{X}_k
		+ B_{1,k}^{(m)} u^{(m)}_{1,k} + B_{2,k}^{(m)} u^{(m)}_{2,k}
		+ b_k^{(m)} \bigr]\Delta t \\
		&+ \bigl[ C_{1,k}^{(m)} X^{(m)}_k + C_{2,k}^{(m)}\bar{X}_k
		+ D_{1,k}^{(m)} u^{(m)}_{1,k} + D_{2,k}^{(m)} u^{(m)}_{2,k}
		+ \sigma_k^{(m)} \bigr]\Delta W_k^{(m)} .
	\end{aligned}
\end{equation}
Here, $X_k^{(m)}$ and $u_{i,k}^{(m)}$ ($i=1,2$) denote the state and control at time $t_k$ along the $m$-th sample path. The coefficients $A_{i,k}^{(m)}$, $B_{i,k}^{(m)}$, $C_{i,k}^{(m)}$, and $D_{i,k}^{(m)}$ ($i=1,2$), as well as $b_k^{(m)}$ and $\sigma_k^{(m)}$, are the pathwise realizations of the stochastic coefficients satisfying Assumption~(H1); for instance, $A_{i,k}^{(m)}=A_i(t_k,\omega^{(m)})$. In addition, $
\bar{X}_k=\frac{1}{M}\sum_{m=1}^M X_k^{(m)}$ is the empirical mean over all simulated paths, and $\bar{X}_k\to \EE[X(t_k)]$ as $M\to\infty$ by the law of large numbers.

In the numerical experiments reported in Section \ref{sec:experiments}, the coefficients are sampled per 
scenario and kept fixed along the time grid, i.e.,  $A_{i,k}^{(m)}=A_i^{(m)}$ and similarly for the other system matrices.  This corresponds to a piecewise-constant realization of the underlying  adapted random coefficients and keeps the context dimension manageable.  The same DFPS architecture extends to genuinely time-varying adapted  coefficients $A_i(t_k,\omega^{(m)})$ by including their time-grid realizations, or suitable low-dimensional summaries, in the context variable $\xi$.

\subsubsection{Network Architectures}
Based on the above discretization, the unknown quantities in the discrete leader--follower FBSDE system are approximated by feedforward neural networks. More specifically, the adjoint processes $(Y_k,Z_k)$, the mean-field terms   $\EE[u_i(t)]$ and $\EE[X(t)]$, and the Lagrange multipliers are parameterized by networks referred to as AdjointNets, MacroNets, and LambdaNets, respectively. The corresponding network configurations are summarized in Table~\ref{tab:networks}.

\begin{table}[htpb]
	\centering
	\caption{Network architectures and initialization parameters. All hidden layers use Tanh activation functions.}
	\label{tab:networks}
	\begin{tabular}{llccc}
		\toprule
		\textbf{Network} & \textbf{Input features} & \textbf{Hidden layers} & \textbf{Width} & \textbf{Output gain} \\
		\midrule
		AdjointNet (follower) & $t_k, X_k, \xi, u_2$ & 4 & 128 & 0.05 \\
		AdjointNet (leader)   & $t_k, X_k, \xi$      & 4 & 128 & 0.05 \\
		MacroNet              & $t_k, \xi$           & 4 & 128 & 0.10 \\
		LambdaNet             & $t_k, \xi$           & 3 & 64  & 0.01 \\
		\bottomrule
	\end{tabular}
\end{table}
 Here, $\xi$ denotes the vectorized context consisting of the model coefficients and cost parameters $
\xi=\mathrm{vec}(A_1,A_2,B_1,B_2,C_1,C_2,D_1,D_2,Q_1,Q_2,R_1,R_2,
G_1,G_2,\bar Q_1,\bar Q_2,\bar R_1,\bar R_2). $ By conditioning on the context variable $\xi$, all networks can be trained to accommodate different realizations of the model coefficients within a single framework, thereby avoiding retraining for each individual scenario.

Note that the output layer of each network in Table~\ref{tab:networks}  is linear (without a bounding activation), so in particular the  LambdaNet output is not artificially constrained to a bounded range; the bounded-error condition in Assumption~\ref{ass:bounded_errors}  below is therefore imposed on the inexact dual update rather than on a hard-bounded multiplier range.

\subsubsection{Mean-field Consistency}
A direct Monte Carlo plug-in approximation of the mean-field terms is not adequate in the present setting. The mean-field quantities
$\mathbb{E}[X(t)]$ and $\mathbb{E}[u_i(t)]$ are endogenous equilibrium objects rather than exogenous coefficients. Replacing them with batch-wise empirical averages would externalize these endogenous processes and treat them merely as noisy sample statistics. Such a plug-in treatment may reduce the FBSDE residual
on a given batch, but it does not by itself enforce the fixed-point consistency between the macroscopic mean-field variables and the trajectories induced by the current policies. Therefore, DFPS parameterizes the mean-field processes through MacroNets and employs an augmented Lagrangian mechanism to enforce their
agreement with the empirical Monte Carlo averages.

The outputs of the MacroNets are denoted by $\alpha_i$ and $\beta_i$, which are used to approximate the corresponding mean-field quantities. Since these terms describe macroscopic population behavior, the MacroNets depend only on the time variable $t_k$ and the context $\xi$, and do not take individual sample states as inputs. In the numerical implementation, their outputs are trained to match the empirical averages over simulated trajectories through the consistency conditions
\begin{equation}\label{eq:mf_consist}
	\alpha_i(t_k,\xi)\approx \frac{1}{M}\sum_{m=1}^M u_{i,k}^{(m)},
	\qquad
	\beta_i(t_k,\xi)\approx \frac{1}{M}\sum_{m=1}^M X_k^{(m)}.
\end{equation}
These relations are imposed for all $k=0,\dots,N-1$ and for each
coefficient scenario $\xi$.

\subsubsection{Augmented Lagrangian Formulation and Asymptotic Feasibility}\label{sub:lagrangian}
To enforce the consistency constraints in computation, we adopt an
augmented Lagrangian formulation. Motivated by the relaxation introduced in the theoretical analysis, we associate the constraints with Lagrange multipliers $\lambda_{u,i}$ and $\lambda_{x,i}$ ($i=1,2$), which are parameterized by context-conditional neural networks referred to as LambdaNets. Each LambdaNet defines a mapping $\lambda_{\cdot}(t,\xi;\phi):[0,T]\times\mathbb{R}^{d_c}\to\mathbb{R}^d,$ with inputs $(t,\xi)$, where $d_c$ denotes the dimension of the context vector $\xi$.

For agent $i$, the augmented Lagrangian is given by
\begin{equation}\label{eq:lagrangian}
	\mathcal{L}_i = J_i^N(\hat{u})+\mathcal{L}_N^{(i)}(\theta_i)
	+\langle \lambda_{u,i},\,\bar{u}_i-\alpha_i\rangle_{\Delta t}
	+\langle \lambda_{x,i},\,\bar{X}-\beta_i\rangle_{\Delta t}
	+\frac{\rho_{u,i}}{2}\|\bar{u}_i-\alpha_i\|_{\Delta t}^2
	+\frac{\rho_{x,i}}{2}\|\bar{X}-\beta_i\|_{\Delta t}^2.
\end{equation}
where $J_i^N(\hat{u})$ denotes the discretized empirical cost functional with $\hat u=(\hat u_1,\hat u_2)$ denoting the pair of controls induced by the current network parameters, $\mathcal{L}_N^{(i)}(\theta)$ represents the FBSDE residual loss parameterized by the primal network weights $\theta_i $ (i.e., the weights of the AdjointNets and MacroNets governing the outputs $Y_i, Z_i, \alpha_j$, and $\beta_j$ for $j \in \{1,2\}$), and $\bar{u}_i = \frac{1}{M}\sum_{m=1}^M u_{i}^{(m)}$ is the empirical mean of the control.  The parameters $\rho_{u,i}, \rho_{x,i} > 0$ act as the penalty coefficients. Furthermore, the discrete temporal inner product and its induced norm are defined respectively as
\begin{equation*}
	\langle f, g \rangle_{\Delta t} = \Delta t \sum_{k=0}^{N-1} f_k g_k, \qquad 
	\|f\|_{\Delta t} = \Bigl(\Delta t\sum_{k=0}^{N-1}f_k^2\Bigr)^{1/2}.
	\end{equation*}

The dual variables, parameterized by the LambdaNet weights $\phi$, are updated by minimizing the dual loss:
\begin{equation}\label{eq:dual_update}
	\mathcal{L}_{\lambda}(\phi) = -\langle \lambda(\cdot;\phi),\,\mathrm{viol}\rangle_{\Delta t}
	+\frac{\eta}{2}\|\lambda(\cdot;\phi)-\lambda^{\mathrm{prev}}\|^2,
\end{equation}
where $\mathrm{viol} \in \{\bar{u}_i-\alpha_i, \bar{X}-\beta_i\}$ denotes the corresponding mean-field consistency residual vector, $\lambda^{\mathrm{prev}}$ is the dual snapshot from the preceding optimizer step, and $\eta>0$ is the proximal step-size parameter. The proximal term is included to stabilize the dual update under noisy Monte Carlo gradients.

In the numerical implementation, the penalty coefficients $\rho_{u,i}$ and $\rho_{x,i}$ are scaled by a factor $\tau=1.1>1$ whenever the constraint violation fails to improve by more than $5\%$ relative to the previous Picard iteration. These scalar constraint violations are monitored through the norms
\begin{equation}\label{eq:viol_def}
	V_{u,i} := \|\bar{u}_i-\alpha_i\|_{\Delta t}, \qquad
	V_{x,i} := \|\bar{X}-\beta_i\|_{\Delta t}.
\end{equation}
The multiplicative growth factor $\tau>1$ ensures that the penalty sequence $\{\rho_{v,i}^{(p)}\}$ diverges whenever the stagnation-triggered update is activated infinitely often, providing the asymptotic mechanism formalized in Proposition~\ref{prop:alm_bound} below.

To quantify the effect of inexact dual updates and to formalize the
penalty-induced feasibility mechanism, fix an agent $i$ and a constraint type $v\in\{u,x\}$, and define
\[
r_{u,i}^{(p)}:=\bar{u}_i^{(p)}-\alpha_i^{(p)},\qquad
r_{x,i}^{(p)}:=\bar{X}^{(p)}-\beta_i^{(p)} .
\]
The LambdaNet residual at Picard iteration $p$ is defined as
\begin{equation}\label{eq:eps_net_def}
\varepsilon_{\rm net}^{(p)} := \lambda_{v,i}^{(p+1)}-
\lambda_{v,i}^{(p)}-\rho_{v,i}^{(p)}\,r_{v,i}^{(p)}.
\end{equation}
By construction, $\varepsilon_{\rm net}^{(p)}$ collects all sources of deviation between the actual LambdaNet update and the nominal ALM ascent direction $\rho_{v,i}^{(p)}\,r_{v,i}^{(p)}$, including the LambdaNet parameterization residual, the finite-$N_C$ stochastic gradient descent (SGD) horizon, and the effect of the proximal regularization. Assumption~\ref{ass:bounded_errors} below requires the cumulative effect of these contributions to remain uniformly bounded. We also define $\varepsilon_{\rm opt}^{(p)}:=\|\lambda_{v,i}^{(p+1)}-\lambda_{v,i}^{*,(p)}\|$ 
as the dual subproblem optimization error, where $\lambda_{v,i}^{*,(p)}$ is the exact minimizer of \eqref{eq:dual_update}.

\begin{assumption}\label{ass:bounded_errors}
	There exist constants $\bar\varepsilon_{\rm opt}, \bar\varepsilon_{\rm net}<\infty$ such that, at every iteration $p$, the dual subproblem error and the LambdaNet approximation residual satisfy
	\begin{equation*}
		\varepsilon_{\rm opt}^{(p)}\le \bar\varepsilon_{\rm opt},\qquad
		\|\varepsilon_{\rm net}^{(p)}\|\le \bar\varepsilon_{\rm net}.
	\end{equation*}
\end{assumption}

Under Assumption~\ref{ass:bounded_errors}, the following proposition shows that the constraint violation decreases as the penalty parameters increase.

\begin{proposition}\label{prop:alm_bound}
Let Assumption \ref{ass:bounded_errors} hold. For a fixed constraint type $v\in\{u,x\}$, define the residual norm $\mathcal{R}_{v,i}^{(p)}:=\|r_{v,i}^{(p)}\|$. If the penalty sequence is non-decreasing and the inexact LambdaNet update satisfies \eqref{eq:eps_net_def}, then, for every $p$ with $\rho_{v,i}^{(p)}>\eta^{-1}$, the violation is bounded by
\begin{equation}\label{eq:viol_bound}
	\mathcal{R}_{v,i}^{(p)}\le\frac{\bar\varepsilon_{\rm opt}+\bar\varepsilon_{\rm net}}{\rho_{v,i}^{(p)}-\eta^{-1}}.
\end{equation}
In particular, if the adaptive scheme drives $\rho_{v,i}^{(p)}\to\infty$, then $\mathcal{R}_{v,i}^{(p)}\to 0$.
\end{proposition}

\begin{proof}
	See Appendix~\ref{app:A}.
\end{proof}
Proposition~\ref{prop:alm_bound} establishes the theoretical guarantee that the mean-field consistency constraints are asymptotically satisfied as the penalty parameters increase. While a rigorous a posteriori error analysis---encompassing the neural approximation errors and the contraction of the Picard iteration under random operator-valued coefficients---is omitted due to strict space limitations, the comprehensive numerical convergence of the DFPS framework is extensively validated through the empirical diagnostics presented in Section~\ref{sec:experiments}.

The complete training procedure is summarised 
in Algorithm~\ref{alg:picard}.
\begin{algorithm}[H]
	\caption{Deep FBSDE Picard Solver (DFPS) for the Mean-Field Stackelberg LQ Game}
	\label{alg:picard}
	\begin{algorithmic}[1]
		\Require System matrices, time grid $\pi_N$, sample size $M$, exploratory scenarios $B$, Picard budget $P$, Picard tolerance $\varepsilon_{\mathrm{tol}}$, and inner-step budgets $(N_A, N_B, N_C)$.
		\Ensure Trained follower and leader networks.
		
		\State Initialize the AdjointNets, MacroNets, and LambdaNets; set the LambdaNet outputs to zero.
		
		\Statex \textbf{Stage I: Follower training}
		\State Sample exploratory leader-control scenarios and warm-start the follower MacroNets $(\alpha_F,\beta_F)$.
		\For{$p=0,\ldots,P-1$}
		\State Simulate follower trajectories under a mini-batch of exploratory leader controls.
		\State Update $(Y_1,Z_1)$ by minimizing the follower primal FBSDE residual ($N_A$ steps; MacroNets and LambdaNets fixed).
		\State Update $(\alpha_F,\beta_F)$ by minimizing the follower mean-field consistency loss ($N_B$ steps; AdjointNets and LambdaNets fixed).
	    \If{$V_{u,1}>\varepsilon_{\mathrm{tol}}$ or     $V_{x,1}>\varepsilon_{\mathrm{tol}}$}
		\State Update $(\lambda_{u,1},\lambda_{x,1})$ via a dual augmented-Lagrangian step ($N_C$ steps).
		\EndIf
		\If{the follower relative Picard error is below $\varepsilon_{\mathrm{tol}}$}
		\State \textbf{break}
		\EndIf
		\EndFor
		
		\Statex \textbf{Stage II: Follower response extraction}
		\State Freeze the trained follower networks.
		\State Extract the affine response sensitivities $\mathcal{M}_{1,1,k}$ and $\mathcal{M}_{1,2,k}$ via automatic differentiation of the trained follower response map.
		\State Construct the follower-induced response maps and the aggregated leader coefficients.
		
		\Statex \textbf{Stage III: Leader training}
		\State Warm-start the leader MacroNets $(\alpha_L,\beta_L)$.
		\For{$p=0,\ldots,P-1$}
		\State Simulate leader trajectories and evaluate the aggregated FBSDE generator using the frozen follower response maps.
		\State Update $(Y_2,Z_2)$ by minimizing the leader primal FBSDE residual ($N_A$ steps; MacroNets and LambdaNets fixed).
		\State Update $(\alpha_L,\beta_L)$ by minimizing the leader mean-field consistency loss ($N_B$ steps; AdjointNets and LambdaNets fixed).
		\If{$V_{u,2}>\varepsilon_{\mathrm{tol}}$ or $V_{x,2}>\varepsilon_{\mathrm{tol}}$}
		\State Update $(\lambda_{u,2},\lambda_{x,2})$ via a dual augmented-Lagrangian step ($N_C$ steps).
		\EndIf
		\If{the leader relative Picard error is below $\varepsilon_{\mathrm{tol}}$}
		\State \textbf{break}
		\EndIf
		\EndFor
	\end{algorithmic}
\end{algorithm}
 
The algorithm implements the saddle-point structure derived in
Section~\ref{sec2} through an alternating augmented Lagrangian scheme. Within each Picard iteration, the AdjointNets are updated by a primal FBSDE regression with the MacroNets and LambdaNets fixed; the MacroNets are then updated by a mean-field consistency regression with the AdjointNets and LambdaNets fixed; and, when the active-set condition is triggered, the LambdaNets are updated by
a dual-ascent step driven by the current consistency residual. Hence, the numerical procedure preserves the primal--dual structure of the theoretical optimality system while remaining implementable through block-coordinate neural optimization.

In the reported experiments, the inner optimization budgets are set to $N_A=600$, $N_B=600$, and $N_C=50$ gradient steps, respectively, together with a $500$-step MacroNet warm-start. The simulations use $N=100$ time intervals, $M=64$ sample paths per scenario, $B=48$ exploratory environments, and at most $P=20$ Picard iterations.

The DFPS framework is compatible with an optional joint Stackelberg refinement loop after Stage~III, in which the follower response map and the leader policy are updated in a fully coupled manner. In the present implementation, this refinement is not activated because the sequential extraction-and-training procedure already reaches the prescribed residual tolerance $\varepsilon_{\mathrm{tol}}$ across all reported scenarios. Thus, the sequential structure provides the desired numerical accuracy while avoiding the
additional overhead of a fully coupled bilevel refinement.

\paragraph{Computational cost}
The full DFPS pipeline involves approximately $25{,}000$ inner gradient steps per scenario. The reported experiments were completed in about 6 hours on a single NVIDIA Tesla T4 GPU. This offline cost is effectively amortized: once trained, the context-conditional framework solves new coefficient realizations without requiring retraining.

\subsection{Numerical Experiments}
\label{sec:experiments}
The numerical study is designed to validate DFPS from six complementary perspectives: numerical convergence and feasibility, discretization sensitivity, a Riccati sanity check in the constant-coefficient regime, ablation analysis, empirical Stackelberg optimality, and financial interpretation. Accordingly,
the experiments examine FBSDE residuals and mean-field consistency constraints, temporal refinement under constant and random coefficients, agreement with a classical Riccati baseline when such a baseline is available, the roles of response-sensitivity extraction and augmented Lagrangian enforcement, unilateral-deviation stability, and the economic implications of the stochastic portfolio application.

To evaluate the DFPS framework under random operator-valued coefficients, the state dynamics and cost parameters are independently sampled from the distributions specified in Table~\ref{tab:coeff_ranges}. The initial state is drawn from $X_0 \sim \mathcal{N}(0,\,0.1\,I_n)$. The Monte Carlo sample size is fixed at $M=64$; preliminary sensitivity analyses across $M\in\{16,32,64,128,256\}$ confirm that the empirical mean-field approximations robustly stabilize at this configuration, thereby mitigating the need for excessively large mini-batches.

\begin{table}[h]
	\centering
	\caption{Distributions and dimensions of the randomized state and cost coefficients.}
	\label{tab:coeff_ranges}
	\begin{tabular}{llc}
		\toprule
		\textbf{Coefficient} & \textbf{Distribution} & \textbf{Dimension} \\
		\midrule
		$A_1$ & Diagonal $\sim \mathcal{U}[-1.0, -0.4]$ & $n\times n$ \\
		$A_2$ & $\mathcal{U}[0.1, 0.4]$ & $n\times n$ \\
		$B_1$ & $\mathcal{U}[0.7, 1.3]$ & $n\times m_1$ \\
		$B_2$ & $\mathcal{U}[0.3, 0.8]$ & $n\times m_2$ \\
		$C_1$ & $\mathcal{U}[0.05, 0.15]$ & $n\times n$ \\
		$C_2, D_1, D_2$ & $\mathcal{U}[0.02, 0.08]$ & $n\times n$, $n\times m_i$ \\
		$b, \sigma$ & $\mathcal{U}[0.01, 0.5]$ & $\mathbb{R}^n$ \\
		$Q_i, R_i, G_i$ & Diagonal $\sim\mathcal{U}[0.99, 1.01]$ & $n\times n$, $m_i\times m_i$ \\
		$\bar{Q}_i, \bar{R}_i$ & Diagonal $\sim\mathcal{U}[0.099, 0.101]$ & $n\times n$, $m_i\times m_i$ \\
		\bottomrule
	\end{tabular}
\end{table}

\subsubsection{Numerical Convergence and Feasibility}
\label{ssec:convergence}

Figure~\ref{fig:training} presents the primary convergence diagnostics of the DFPS algorithm. Since the coefficients are random and the follower is trained under exploratory leader-control scenarios, the empirical cost is not expected to decrease
monotonically along Picard iterations. The relevant numerical question is therefore not monotone descent of $J_1$, but whether the optimality residuals and the mean-field consistency violations are driven to the prescribed tolerance while the realized costs remain statistically stable. This is precisely what is observed in Figure~\ref{fig:training}. Although $J_1$ exhibits moderate oscillations, the last ten Picard iterates concentrate
around a stable operating level, with mean $0.354$ and standard deviation $0.006$. In contrast, the follower BSDE residual decays by approximately three orders of magnitude and reaches $\mathcal E_1=3\times 10^{-4}$, while the leader BSDE residual reaches $\mathcal E_2=2\times 10^{-4}$ within five Picard iterations. The terminal mismatch
\[ \delta_T = \frac{\mathbb{E}[\|Y_1(T)-G_1X(T)\|]}
{\mathbb{E}[\|G_1X(T)\|]} =
1.20\%
\]
further indicates that the terminal condition of the follower adjoint equation is satisfied to a small relative error.

Panel~(c) illustrates the follower mean-field consistency violations under the adaptive augmented Lagrangian scheme. Both follower violations terminate below the prescribed tolerance $\varepsilon_{\mathrm{tol}}=0.02$. A more detailed view of the penalty adaptation and all four violations is provided in Figure~\ref{fig:rho_violation}.

The robustness of the training procedure is further validated across three independent random seeds $(42,123,248)$. The resulting coefficients of variation are low at $1.1\%$ for $J_1$ and $0.5\%$ for $J_2$, with $J_1=0.351\pm0.004$ and $J_2=0.211\pm0.001$. The maximum global violation norm
\begin{equation}
\|\nu\|_{\max}:=	\max(V_{u,1},V_{x,1},V_{u,2},V_{x,2})
=	0.0132\pm0.0005<\varepsilon_{\mathrm{tol}}=0.02
\end{equation}
also remains uniformly below the feasibility threshold. These observations provide empirical evidence that the sequential follower-response extraction and leader Picard training achieve stable residual and feasibility accuracy in the tested scenarios. Since the prescribed tolerances are already reached, we
keep the optional fully coupled Stackelberg refinement disabled in the reported experiments to reduce computational overhead.

\begin{figure}[h]
	\centering
	\includegraphics[width=0.9\linewidth]
	{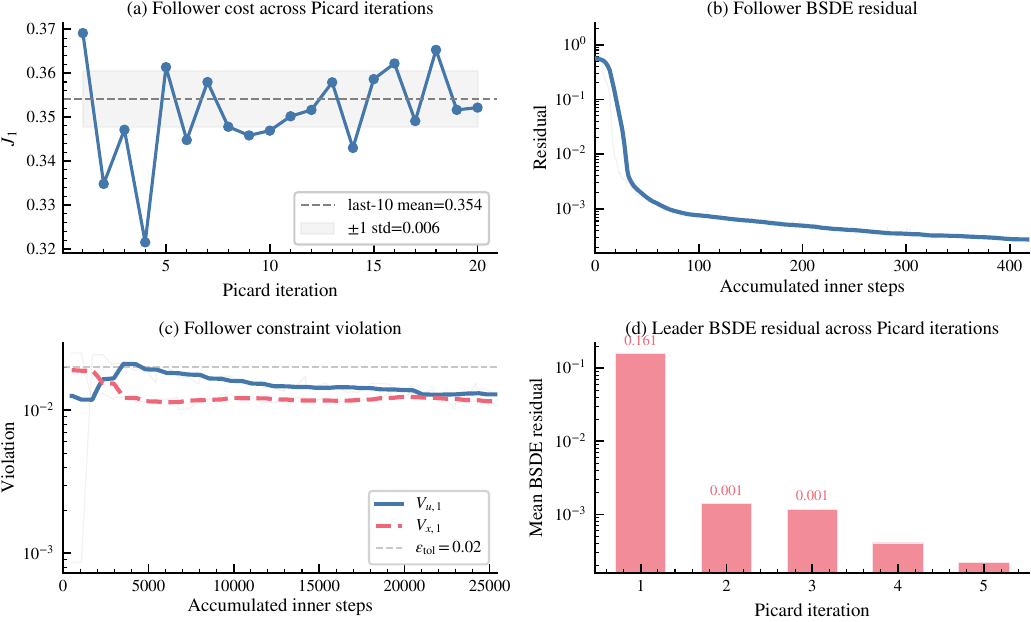}
	\caption{Training convergence of the 
		DFPS algorithm: (a)~follower cost 
		$J_1$ across Picard iterations 
		(dashed line: last-10 mean $=0.354$, 
		shaded band: $\pm1$ std), 
		(b)~follower BSDE residual, 
		(c)~mean-field constraint violations 
		$V_{u,1}$ and $V_{x,1}$ with tolerance 
		$\varepsilon_{\rm tol}=0.02$, 
		(d)~leader BSDE residual across Picard 
		iterations.}
	\label{fig:training}
\end{figure}

Figure~\ref{fig:rho_violation} further reports the coupled evolution of the mean-field consistency violations and the adaptive penalty parameters. The penalties are increased only when the corresponding violation stagnates. Despite transient early-stage increases, all four violations are eventually driven below $\varepsilon_{\mathrm{tol}}=0.02$, with final values
$V_{u,1}=0.0126$, $V_{x,1}=0.0107$, $V_{u,2}=0.0062$, and
$V_{x,2}=0.0047$. The follower penalties grow more substantially
$(\rho_{u,1}^{\mathrm{final}}=0.157,\rho_{x,1}^{\mathrm{final}}=0.314)$, whereas the leader penalties remain close to their initial values $(\rho_{u,2}^{\mathrm{final}}=0.055,\rho_{x,2}^{\mathrm{final}}=0.110)$. This supports the feasibility mechanism in Proposition~\ref{prop:alm_bound} and indicates that the follower-response extraction yields a well-conditioned
leader initialization, reducing the need for aggressive penalty adaptation.

\begin{figure}[h]
	\centering
	\includegraphics[width=0.95\linewidth]
	{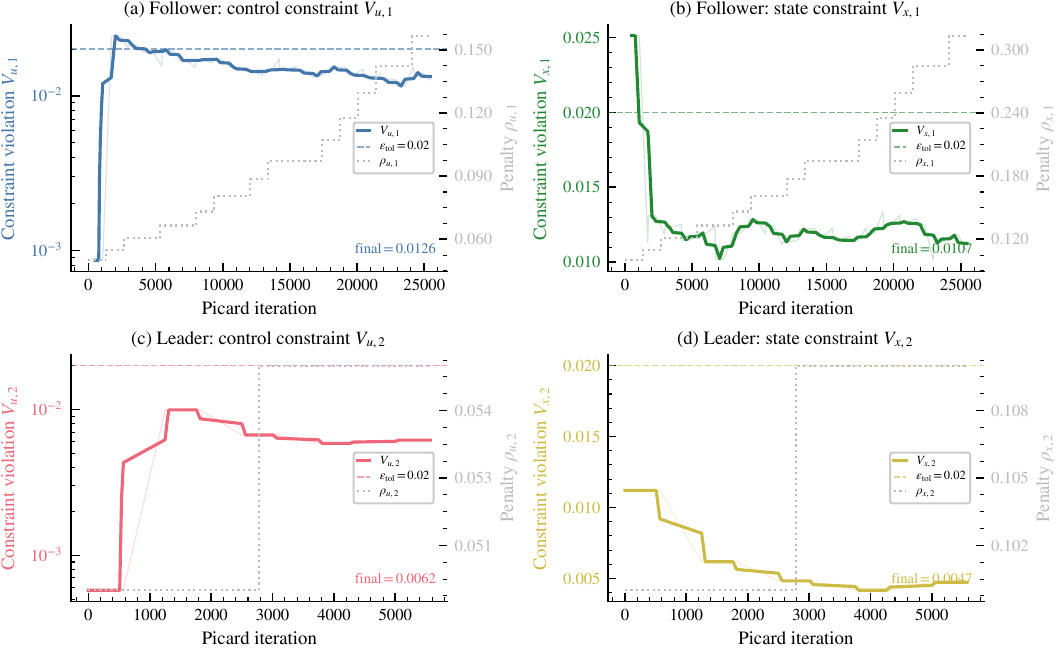}
	\caption{Adaptive ALM diagnostics: constraint violations 
		(left axes, log scale) and penalty parameters $\rho$ 
		(right axes, dotted) over accumulated inner optimization steps. 
		Dashed horizontal line: tolerance 
		$\varepsilon_{\mathrm{tol}}=0.02$. 
		All four violations terminate below tolerance, consistent with 
		Proposition~\ref{prop:alm_bound}.}
	\label{fig:rho_violation}
\end{figure}

\subsubsection{Discretization Sensitivity}

\paragraph{Temporal discretization convergence}
We first study the effect of the number of time steps $N$ in a setting where a Riccati reference solution is available. For this purpose, the experiment is conducted under constant coefficients, for which the Riccati ODE provides the reference follower cost $J_1^\ast=0.225$.

Figure~\ref{fig:convergence} examines the effect of temporal refinement on the solution quality. Panel~(a) shows that the follower cost $J_1^N$ decreases from $0.521$ at $N=10$ toward the Riccati reference $J_1^\ast=0.225$, achieving a relative error of $3.4\%$ at $N=200$. Panel~(b) shows that the BSDE residual
decays from $2.0\times10^{-2}$ to $1.1\times10^{-4}$ over the same range. Panel~(c) plots the self-convergence error $|J_1^N-J_1^{N_{\max}}|$ against $\Delta t$ on a log-log scale. The fitted slope of $1.30$ is compatible with first-order temporal convergence of the Euler--Maruyama discretization. The
error measured against the Riccati reference yields a comparable fitted slope of $1.20$. Monte Carlo stability is assessed over $20$ independent replicates at $M=64$ paths, yielding $J_1=0.2458\pm0.0021$ with standard error $4.7\times10^{-4}$.

\begin{figure}[H]
	\centering
	\includegraphics[width=\textwidth]{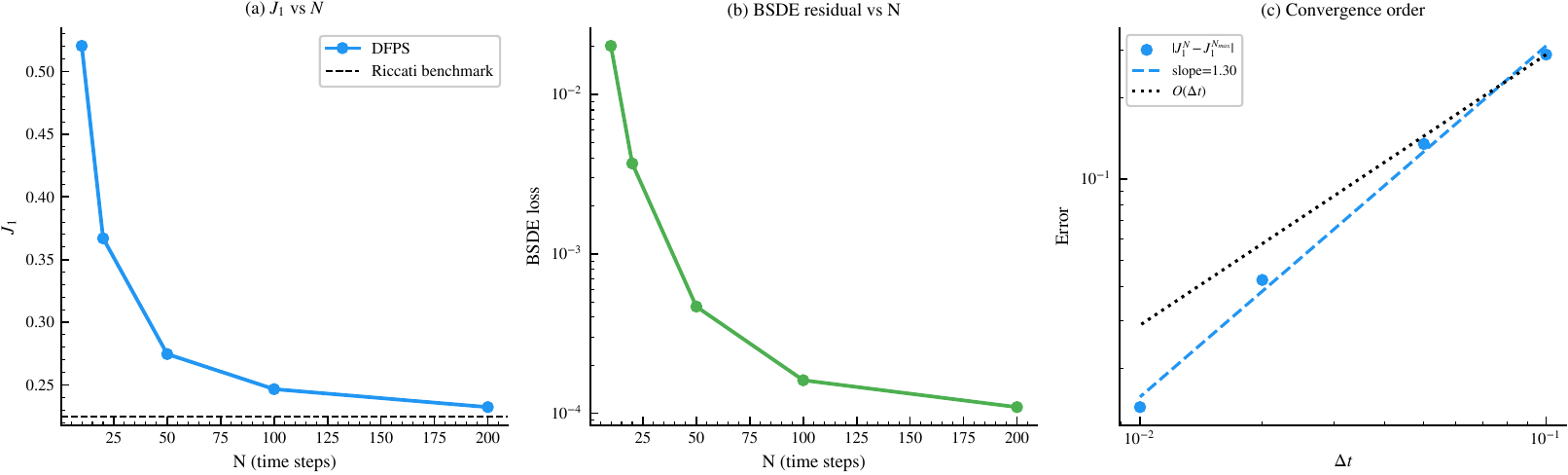}
	\caption{Temporal discretization convergence (constant-coefficient setting):
		(a)~follower cost $J_1^N$ vs.\ Riccati reference $J_1^\ast=0.225$,
		(b)~BSDE residual, (c)~self-convergence error on log-log scale with fitted
		slope $1.30$.}
	\label{fig:convergence}
\end{figure}

\paragraph{Self-convergence under random coefficients}
Under stochastic operator-valued coefficients---the primary setting of this paper---no closed-form reference solution is available. Table~\ref{tab:convergence} therefore reports self-convergence with respect to $N$, using $N=200$ as the numerical reference. The relative error decreases from $18.2\%$ at $N=50$ to $6.2\%$ at $N=100$, which is consistent with the temporal refinement behavior observed in the constant-coefficient test.

\begin{table}[h]
	\centering
	\caption{DFPS self-convergence in $N$ under random coefficients
		(reference: $N=200$, $J_1=0.2323$).}
	\label{tab:convergence}
	\begin{tabular}{ccc}
		\toprule
		$N$ & $J_1$ & Relative error \\
		\midrule
		50  & 0.2746 & $18.2\%$ \\
		100 & 0.2466 &  $6.2\%$ \\
		200 & 0.2323 &  $0.0\%$ (ref) \\
		\bottomrule
	\end{tabular}
\end{table}

\paragraph{Computational scaling}
Figure~\ref{fig:scalability} reports a profiling experiment for the dependence of DFPS on the state dimension $n$. Panel~(b) shows that the number of trainable parameters grows approximately as $\mathcal{O}(n^{1.37})$ ($R^2=0.905$), reflecting the polynomial growth of the network input layers. Panel~(a) reports the wall-clock time for a fixed 200-epoch warm-up, which remains nearly constant across the tested dimensions and is dominated by fixed GPU overhead in this profiling regime. Panel~(c) contrasts this polynomial parameter growth with the theoretical FDM grid size $G^n$ ($G=20$), which grows exponentially and exceeds $10^{65}$ at $n=50$. These profiling results do not constitute a
full convergence-complexity analysis for large $n$, since asymptotic
convergence for $n\geq10$ would require longer training. Nevertheless, they suggest that the DFPS architecture avoids the exponential grid growth associated with grid-based finite-difference methods. Note that FDM does not apply to the random-coefficient setting; it  is shown here only as a reference for the curse of dimensionality.

\begin{figure}[h]
	\centering
	\includegraphics[width=0.95\linewidth]{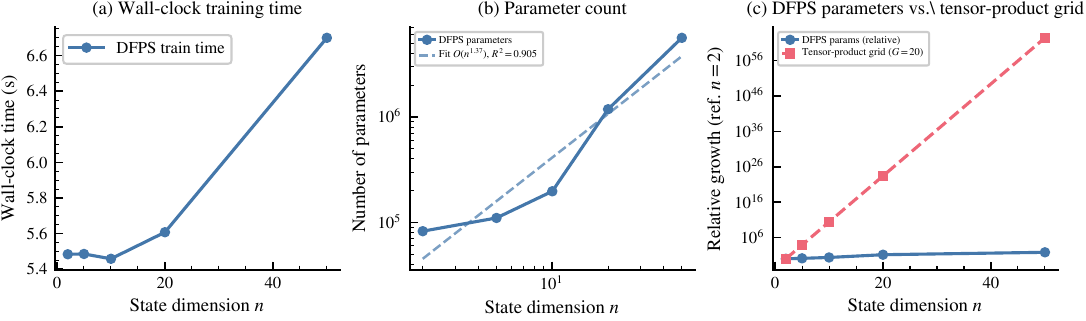}
	\caption{Computational scaling with state dimension $n$: (a)~wall-clock
		training time for a 200-epoch warm-up, (b)~parameter count fitted as
		$\mathcal{O}(n^{1.37})$, (c)~relative growth of DFPS parameters vs.\
		theoretical FDM grid size $G^n$ ($G=20$). FDM does not apply to the
		random-coefficient setting and is shown only for reference.}
	\label{fig:scalability}
\end{figure}
\subsubsection{Baseline Calibration: Riccati Sanity Check}
\label{ssec:sanity_check}

Under constant coefficients, the follower subproblem reduces to a classical LQ problem for which a Riccati reference solution is available. We emphasize that this constant-coefficient regime serves as a qualitative sanity check rather than as the primary benchmark for the proposed method. Table~\ref{tab:riccati}
reports the results across three distinct constant-coefficient scenarios. The observed mean relative error of $12.4\%$ can be partly attributed to the distributional shift in this test setup: DFPS is trained over random coefficient scenarios drawn from the broad ranges in Table~\ref{tab:coeff_ranges}, whereas the Riccati comparison evaluates constant coefficient test instances. Notably, no targeted fine-tuning for these constant coefficients is performed.

This deterministic regime is only a secondary validation. Under generic stochastic operator-valued coefficients, which constitute the primary focus of this paper, the associated Riccati approach does not provide a tractable closed-form baseline \cite{wei2019linear}. In this generalized stochastic
regime, we validate the computed equilibrium through residual diagnostics and the unilateral deviation tests presented in Section~\ref{ssec:equilibrium}, neither of which requires an explicit reference solution.

\begin{table}[h]
	\centering
	\caption{Follower cost: DFPS vs.\ Riccati reference in constant-coefficient scenarios.}
	\label{tab:riccati}
	\begin{tabular}{lccc}
		\toprule
		Scenario & $J_1^{\mathrm{Riccati}}$ & $J_1^{\mathrm{DFPS}}$ & Rel.\ error \\
		\midrule
		1 & 0.1990 & 0.2341 & $17.6\%$ \\
		2 & 0.2475 & 0.2854 & $15.3\%$ \\
		3 & 0.3895 & 0.3727 &  $4.3\%$ \\
		\midrule
		Mean & 0.2787 & 0.2974 & $12.4\%$ \\
		\bottomrule
	\end{tabular}
\end{table}

\subsubsection{Ablation Study}
\label{ssec:ablation}

Table~\ref{tab:ablation} presents an ablation study designed to isolate the essential structural ingredients of the DFPS framework: response-sensitivity extraction, phase-separated Stackelberg training, and augmented Lagrangian enforcement of mean-field consistency. 

Recall that the sensitivity operator $\mathcal{M}_{1,2}$ characterizes how the follower's optimal response varies with respect to the leader's control. By masking this bilevel sensitivity (i.e., manually enforcing $\mathcal{M}_{1,2}\equiv 0$), the leader effectively ignores the follower's rational response mechanism, degrading the system into a simultaneous-play Nash approximation. While the follower's cost remains empirically stable in this specific scenario, the leader's cost drastically increases by $49.3\%$. This performance gap indicates that the explicit Stackelberg anticipation mechanism is essential for leader-side optimality.

The second variant trains both agents simultaneously without phase separation (Naive Deep BSDE, cf.\ \cite{han2017deep}). Although the empirical training loss  converges to a minimal value of $5.87\times10^{-4}$, the resulting actual costs increase by $18.4\%$ for $J_1$ and $34.7\%$ for $J_2$. This shows that  a small BSDE residual alone does not guarantee the recovery of the proper sequential Stackelberg structure. Finally, removing the augmented Lagrangian method (No ALM) leads to complete training divergence, suggesting that explicit enforcement of mean-field consistency constraints is crucial for numerical stability.

\begin{table}[t]
	\centering
	\small 
	\caption{Ablation study (mean $\pm$ std.\ over 5 independent runs, 
		$N=100$, $M=64$). Positive $\Delta J_i$ indicates cost degradation 
		relative to Full DFPS.}
	\label{tab:ablation}
\begin{tabular}{lcccc}
	\toprule
	\textbf{Variant} & $J_1$ & $\Delta J_1$ & $J_2$ & $\Delta J_2$ \\
	\midrule
	Full DFPS$^a$
	& $\mathbf{0.3475\pm0.011}$ & ---
	& $\mathbf{0.2114\pm0.008}$ & --- \\
	No bilevel$^b$
	& $0.3475\pm0.011$ & $0.0\%$
	& $0.3156\pm0.009$ & $+49.3\%$ \\
	Naive Deep BSDE$^c$
	& $0.4114\pm0.010$ & $+18.4\%$
	& $0.2847\pm0.011$ & $+34.7\%$ \\
	No ALM
	& \multicolumn{4}{c}{Does not converge} \\
	\bottomrule
	\multicolumn{5}{l}{$^a$ Full hierarchical model.}\\
	\multicolumn{5}{l}{$^b$ $\mathcal{M}_{1,2}\equiv 0$, 
		response sensitivity removed.}\\
	\multicolumn{5}{l}{$^c$ No phase separation.}\\
\end{tabular}
\end{table}
\subsubsection{Equilibrium Validation and Financial Application}
\label{ssec:equilibrium}

\paragraph{Numerical Stability under Control Perturbations}
To assess the numerical stability of the computed Stackelberg solution under control perturbations, we test the strategy pair 
$(\tilde u_1, \tilde u_2)$ against random unilateral deviations
$u_i^{\mathrm{dev}} = \tilde u_i + \varepsilon\,\delta u_i$,
where $\delta u_i$ is drawn from random unit-norm directions in the 
control space, and $\varepsilon$ measures the perturbation magnitude. To preserve the hierarchical commitment structure of the Stackelberg game, any leader deviation $u_2^{\mathrm{dev}}$ is evaluated together with a recomputed follower response.

Figure~\ref{fig:deviation} reports the cost increment 
$\Delta J_i(\varepsilon) = J_i(u_i^{\mathrm{dev}}, \tilde{u}_{-i}) - J_i^\ast$ as a function of the perturbation magnitude $\varepsilon \in [-2, 2]$, averaged over $6$ seeds and $32$ random directions per seed. The shaded grey region marks the $\pm 1\%$ numerical tolerance band. The baseline equilibrium costs are $J_1^\ast = 0.2312$ and $J_2^\ast = 0.1294$. The maximum absolute deviations are $0.0014$ for both agents, corresponding to $0.60\%$ of $J_1^\ast$ and $1.07\%$ of $J_2^\ast$, both comparable to the prescribed $\pm 1\%$ tolerance. These small residual deviations are consistent with the representational resolution of the neural approximation. The stability of both costs under random unilateral perturbations supports the practical reliability of the computed Stackelberg solution.

\begin{figure}[htpb]
	\centering
	\includegraphics[width=\columnwidth]{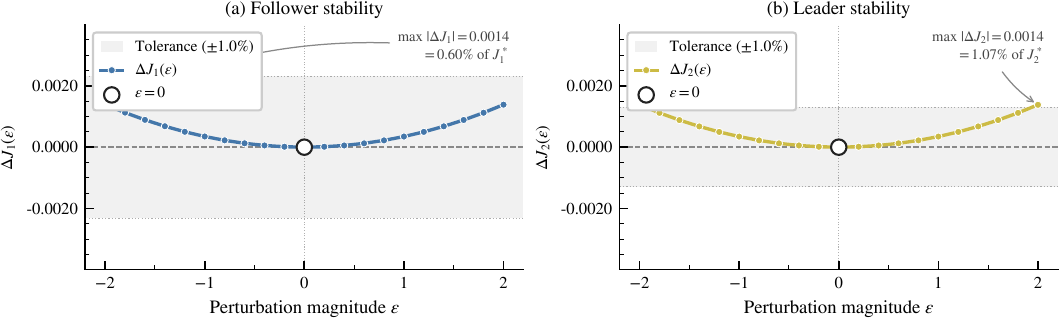}
	\caption{Numerical stability test: cost increment 
		$\Delta J_i(\varepsilon)$ vs.\ perturbation magnitude $\varepsilon$ 
		along random unit-norm directions, averaged over $6$ seeds and 
		$32$ directions per seed. Grey region: $\pm 1\%$ numerical 
		tolerance. The maximum relative deviations are $0.60\%$ (follower) 
		and $1.07\%$ (leader), both within tolerance.}
	\label{fig:deviation}
\end{figure}

\paragraph{Financial interpretation under stochastic volatility}
We illustrate the practical relevance of DFPS using a mean-variance portfolio Stackelberg game. A fund manager (leader, agent~2) sets a strategic benchmark allocation, while an individual investor (follower, agent~1) adjusts her trading strategy to track it under stochastic market conditions. The two-dimensional state $X=(X_1,X_2)^\top\in\mathbb{R}^2$ represents deviations
in stock holdings and cash positions relative to the benchmark. The scalar controls $u_1,u_2\in\mathbb{R}$ denote the trading rates of the investor and the manager, respectively.

Unlike the generic distributions in Table~\ref{tab:coeff_ranges} used for general convergence diagnostics, the financial scenario adopts the specialized distributions described below to reflect the asymmetry between investor and manager.
The system matrices follow the structure
\begin{equation*}
	A_1 = \begin{pmatrix} -\kappa & 0 \\ 0 & -0.3 \end{pmatrix}, \quad
	A_2 = \begin{pmatrix} \gamma & 0 \\ 0 & 0.1 \end{pmatrix}, \quad
	B_i = \ell_i \begin{pmatrix} 1 \\ -\rho_i \end{pmatrix},
\end{equation*}
with mean-reversion rate $\kappa = 0.50$ acting on the 
position-deviation component, herding coefficient $\gamma = 0.15$ 
capturing cross-sectional coupling on the same component, and 
budget-constraint ratios $\rho_1 = 0.9$, $\rho_2 = 0.8$ encoding 
the partial cash offset of stock trades. The liquidity coefficients 
are sampled per scenario as $\ell_1 \sim U[0.7, 1.3]$ for the investor  and $\ell_2 \sim U[1.2, 2.0]$ for the manager, reflecting the larger  market impact of institutional trades. The volatility coefficient  $\sigma$ is drawn from $U[0.04, 0.10]$, $U[0.10, 0.20]$, or 
$U[0.18, 0.30]$ for the low-, medium-, and high-volatility regimes, 
respectively.

Consistent with the randomized training protocol in 
Section~\ref{sub:dis}, the cost parameters are sampled from 
$U[0.99 \cdot \mathrm{nom}, 1.01 \cdot \mathrm{nom}]$ around the 
following nominal values:
\begin{align*}
	Q_1^{\mathrm{nom}} &= \operatorname{diag}(3,1), &
	R_1^{\mathrm{nom}} &= 0.5, &
	G_1^{\mathrm{nom}} &= \operatorname{diag}(2,0.5), &
	\bar{Q}_1^{\mathrm{nom}} &= \operatorname{diag}(0.3,0.1), \\
	Q_2^{\mathrm{nom}} &= \operatorname{diag}(2,0.5), &
	R_2^{\mathrm{nom}} &= 1.0, &
	G_2^{\mathrm{nom}} &= \operatorname{diag}(1.5,0.3), &
	\bar{Q}_2^{\mathrm{nom}} &= \operatorname{diag}(0.5,0.2).
\end{align*}
The asymmetry $Q_1^{\mathrm{nom}}\succ Q_2^{\mathrm{nom}}$ in the
position-tracking component reflects the investor's stronger tracking incentive, while $R_2^{\mathrm{nom}}=2R_1^{\mathrm{nom}}$ captures the manager's larger institutional trading friction. The randomized coefficients introduce variability in both the market dynamics and the objective weights.

Table~\ref{tab:finance} reports the equilibrium costs under different  volatility regimes. The dominant effect is the gap between the  deterministic and stochastic regimes for the investor's tracking  cost ($J_1$: $0.42$ vs $\approx 0.50$, a gap of approximately $2.9$  Monte Carlo standard deviations), indicating that the presence of  volatility uncertainty, rather than its magnitude, is the primary  driver of the investor's tracking-cost increase. The manager's cost  $J_2$ is comparatively insensitive to the presence of stochastic volatility, with the deterministic and stochastic regimes differing by less than one standard deviation. Within the stochastic regimes, the variation across volatility levels for both costs remains within one Monte Carlo standard deviation and should be interpreted as a 
qualitative trend rather than a statistically significant ordering.

\begin{table}[htpb]
	\centering
	\caption{Equilibrium costs under varying volatility regimes 
		(mean $\pm$ Monte Carlo standard deviation over 5 independent 
		path realizations; coefficient instances are fixed per regime).}
	\label{tab:finance}
	\begin{tabular}{lcc}
		\toprule
		Scenario & $J_1$ (Investor) & $J_2$ (Manager) \\
		\midrule
		Low vol ($\sigma \sim U[0.04, 0.10]$)    & $0.5025\pm0.028$ & $0.1721\pm0.029$ \\
		Medium vol ($\sigma \sim U[0.10, 0.20]$) & $0.5042\pm0.028$ & $0.1731\pm0.030$ \\
		High vol ($\sigma \sim U[0.18, 0.30]$)   & $0.5081\pm0.029$ & $0.1752\pm0.030$ \\
		Deterministic baseline (medium-vol nominal $\sigma=0.15$) & $0.4232\pm0.028$ & $0.1690\pm0.030$ \\
		\bottomrule
	\end{tabular}
\end{table}

Figure~\ref{fig:finance} further illustrates the computed Stackelberg equilibrium trajectory and control strategy. Panel~(a) shows the stock-holding deviation $X_1(t)$ across simulated paths, and Panel~(b) reports the investor's mean optimal trading rate $\mathbb{E}[u_1(t)]$. Panel~(c) highlights the asymmetric impact of stochastic volatility: the investor bears a quantifiable increase in tracking cost (about $19\%$ relative to the deterministic baseline), while the manager's cost remains within Monte Carlo noise. These results demonstrate that hierarchical commitment yields measurable tracking-cost advantages even under operator-valued stochastic volatility, which is a regime where Riccati-based methods are not directly applicable.

\begin{figure}[h]
	\centering
	\includegraphics[width=0.8\linewidth]
	{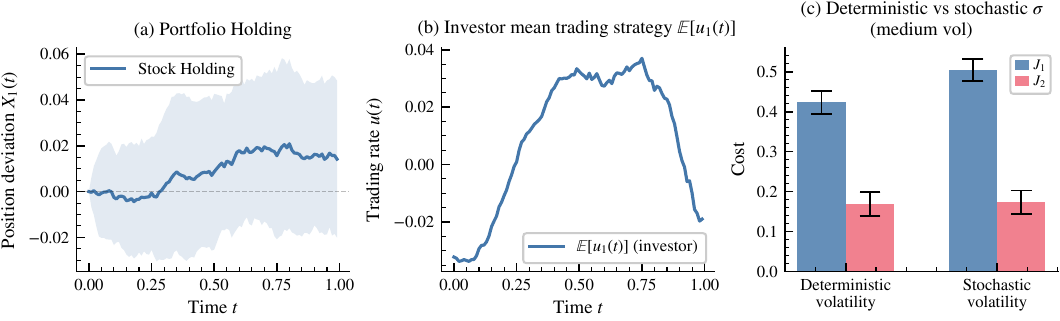}
	\caption{Mean-variance portfolio 
		Stackelberg game 
		($\kappa=0.50$, $\gamma=0.15$): 
		(a)~stock-holding deviation $X_1(t)$ 
		over simulated paths 
		(mean $\pm1\sigma$), 
		(b)~investor's mean optimal trading 
		rate $\mathbb{E}[u_1(t)]$, 
		(c)~equilibrium costs $J_1$ and $J_2$ 
		under deterministic vs.\ stochastic 
		volatility (medium-vol regime, 
		$\sigma\in[0.10,0.20]$).}
	\label{fig:finance}
\end{figure}

\section{Conclusion}\label{sec:conclusion}
This paper developed a theoretical and a deep-learning-based 
numerical method for linear-quadratic mean-field Stackelberg differential games with response-induced stochastic operator-valued coefficients. By applying the extended Lagrange multiplier method, we characterized the Stackelberg equilibrium through a coupled forward-backward stochastic system, in which the optimal controls of both leader and follower admit affine operator-valued representations.

To translate this theoretical structure into a scalable numerical scheme, we proposed the Deep FBSDE Picard Solver (DFPS). Rather than confronting the fully coupled bilevel FBSDE system directly, DFPS employs a phase-separated sequential architecture. This design lets the leader extract the follower's response sensitivity directly from the learned affine response map, avoiding the higher-order variational adjoint equations that typically arise in fully coupled bilevel Stackelberg systems.

Numerical experiments show that DFPS attains small FBSDE residuals 
and mean-field consistency violations under random  coefficients, agrees with the Riccati reference in the constant-coefficient 
sanity check, and passes empirical Stackelberg optimality and 
unilateral-deviation tests. The portfolio application further illustrates the practical relevance of hierarchical commitment under stochastic volatility.

\bibliographystyle{plain}     
\bibliography{ref}

\appendix

\section{The Proof of Problem (MFSOLQ-F)}\label{app:2}

\begin{proof}[The Proof of Theorem \ref{Th:3.1}]
By the linearity of the SDE \eqref{state1} and Lemma~\ref{lem:4.1},
together with the boundedness of all coefficient operators under
(H1), there exist bounded linear operators
$\cH_1:\RR^n\to L^{2,c}_\FF(\RR^n)$,
$\cH_2:\cU_1\to L^{2,c}_\FF(\RR^n)$,
$\cH_3:\cU_2\to L^{2,c}_\FF(\RR^n)$,
and $\cH_0\in L^{2,c}_\FF(\RR^n)$ such that
\[X = \cH_1 x + \cH_2 u_1 + \cH_3 u_2 + \cH_0.
\]
Analogously, there exist bounded linear operators $\cN_i$ ($i=0,1,2,3$) and $\cS_j$ ($j=1,\dots,5$) such that
\[X(T) = \cN_1 x + \cN_2 u_1 + \cN_3 u_2 + \cN_0,
\quad\EE[X] = \cS_1 x + \cS_2 u_1 + \cS_3 u_2 + \cS_4,
\quad\EE[u_1] = \cS_5 u_1.\]
	
Substituting these representations into \eqref{cost0} expresses
$J_1$ as a quadratic functional in $u_1$. After collecting terms,
the quadratic component takes the form $\langle \mathcal{Q}_{u_1} u_1, u_1\rangle_{L^2_\FF(\RR^n)}$, where
\[\mathcal{Q}_{u_1}	= \cH_2^\ast Q_1 \cH_2 + \cS_2^\ast \bar{Q}_1 \cS_2+ R_1 + \cS_5^\ast \bar{R}_1 \cS_5+ \cN_2^\ast G_1 \cN_2.
\]
By Assumption (H2), $R_1 \ge \delta I_{m_1}$ a.e., and all other
summands are nonnegative operators. Hence,
\[\langle \mathcal{Q}_{u_1} u_1, u_1\rangle_{L^2_\FF(\RR^n)}
\ge \delta\, \EE\int_0^T |u_1(s)|^2\,\drm s,
\]which establishes strict convexity.
\end{proof}

\begin{proof}[The Proof of Theorem \ref{Th:3.2}]
The existence of an optimal control follows from a standard application of Mazur's theorem along the lines of \cite[Theorem~5.2]{yong1999stochastic}, utilizing the strong lower-semicontinuity of $J_1$ with respect to $u_1$ implied by (H2).
	
Furthermore, by Theorem \ref{Th:3.1}, the cost functional $J_1(u_1(\cdot), u_2(\cdot))$ is strictly convex in $u_1(\cdot)$. This structural property immediately guarantees that the optimal control, whose existence is established above, must be unique.
	 
Now, we prove the necessity of the stationary condition \eqref{eq:so}.
Let $\tilde{u}_1$ be optimal with the associated state $\tilde{X}$, and let $u_1^\epsilon = \tilde{u}_1 + \epsilon v$ for an arbitrary $v\in\mathcal{U}_1[0,T]$ and $\epsilon\in\mathbb{R}$. The corresponding state perturbation $X_1 := \epsilon^{-1}(X^\epsilon - \tilde{X})$ solves the variational equation
\begin{equation}\label{eq:pf3.2-0}
	\left\{
	\begin{aligned}
		\drm X_1(s) &= [\cA_1 X_1 + \cA_2 \bar{X}_1 + \cB_1 v]\,\drm s
		+ [\cC_1 X_1 + \cC_2 \bar{X}_1 + \cD_1 v]\,\drm W(s),\\
		X_1(0) &= 0.
	\end{aligned}
	\right.
\end{equation}
By Lemma \ref{lem:4.1}, $\mathbb{E}[\sup_{s\in[0,T]} |X^\epsilon(s) - \tilde{X}(s)|^2] \le K\epsilon^2$.

The first-order optimality condition yields
\begin{equation}\label{eq:pf3.2-1}
	\begin{aligned}
		0 &= \lim_{\epsilon \to 0}
		\frac{J_1(u_1^\epsilon,u_2) - J_1(\tilde{u}_1,u_2)}{\epsilon}\\
		&= 2\,\mathbb{E}\left\{\int_0^T \left[
		\langle Q_1 \tilde{X} +\mathbb{E}[\bar{Q}_1]\bar{\tilde{X}},\,X_1\rangle
		+ \langle R_1 \tilde{u}_1 +\mathbb{E}[\bar{R}_1]\bar{\tilde{u}}_1,\,v\rangle
		\right]\drm s
		+ \langle G_1 \tilde{X}(T),\,X_1(T)\rangle\right\},
	\end{aligned}
\end{equation}
where the last equality follows from the fact that for any deterministic process $\eta(\cdot)$, $\mathbb{E}\langle \bar{Q}_1 \eta, \bar{X}_1 \rangle = \langle \mathbb{E}[\bar{Q}_1] \eta, \mathbb{E}[X_1] \rangle = \mathbb{E}\langle \mathbb{E}[\bar{Q}_1]\eta, X_1 \rangle$.

Then, applying It\^o's formula to $\langle \tilde{Y}, X_1\rangle$ and taking expectations, the drift terms of the BSDE \eqref{eq:F1} exactly cancel the
$Q_1$- and $\bar{Q}_1$-terms in \eqref{eq:pf3.2-1}, leading to
\begin{equation}\label{eq:pf3.2-2}
	\mathbb{E}\langle G_1\tilde{X}(T), X_1(T)\rangle
	= \mathbb{E}\int_0^T\Bigl[
	-\langle Q_1\tilde{X} +\mathbb{E}[\bar{Q}_1]\bar{\tilde{X}},\,X_1\rangle
	+ \langle \cB_1^\ast \tilde{Y} + \cD_1^\ast \tilde{Z},\,v\rangle
	\Bigr]\drm s.
\end{equation}
Substituting \eqref{eq:pf3.2-2} into \eqref{eq:pf3.2-1} gives
\[
0 = 2\,\mathbb{E}\int_0^T
\langle \cB_1^\ast \tilde{Y} + \cD_1^\ast \tilde{Z}
+ R_1\tilde{u}_1 + \mathbb{E}[\bar{R}_1]\bar{\tilde{u}}_1,\,v\rangle\,\drm s.
\]
Since $v\in\mathcal{U}_1[0,T]$ is an arbitrary adapted process, the stationarity condition \eqref{eq:so} follows a.e.\ $s \in [0,T]$, $\mathbb{P}$-a.s.

For the sufficiency of \eqref{eq:so}, since $u_1^\epsilon = \tilde{u}_1 + \epsilon v$, combining with the stationary condition \ref{eq:so}, then we obtain
\[
J_1(u_1^\epsilon,u_2) - J_1(\tilde{u}_1,u_2)
= \epsilon^2 I,
\]
where  
\[
I = \mathbb{E}\int_0^T\bigl[
\langle Q_1 X_1, X_1\rangle + \langle \bar{Q}_1 \bar{X}_1, \bar{X}_1\rangle
+ \langle R_1 v, v\rangle + \langle \bar{R}_1 \bar{v}, \bar{v}\rangle
\bigr]\drm s + \mathbb{E}\langle G_1 X_1(T), X_1(T)\rangle \ge 0
\]
by (H2). Hence $J_1(u_1^\epsilon,u_2) \ge J_1(\tilde{u}_1,u_2)$ for
all $\epsilon$ and $v$, establishing that $\tilde{u}_1$ is indeed optimal.
\end{proof}

Now, we  present the detailed proof of Lemma \ref{lem:3.7} in Sub-Problem (F-2).

\begin{proof}[The proof of Lemma \ref{lem:3.7}]
	Let $\bl_1^\ast = (\lambda^{\ast}_1, \ti{\lambda}^{\ast}_1)$ be the optimal pair to Problem (F-2), and let $(X^{\bea_1,\bl_1^\ast}(\cdot), Y^{\bea_1,\bl_1^\ast}(\cdot), Z^{\bea_1,\bl_1^\ast}(\cdot))$ be the corresponding state process satisfying the FBSDE \eqref{eq:so-u} with $(\la_1,\ti{\la}_1)$ replaced by $(\la_1^\ast, \ti{\la}_1^\ast)$.  
	
	Define $\bl_1^\ep = (\lambda^\ep_1, \ti{\lambda}^\ep_1)$ by $\lambda^\ep_1 = \lambda^\ast_1 + \epsilon\lambda^1_1$ and $\ti{\lambda}^\ep_1 = \ti{\lambda}^\ast_1 + \epsilon\ti{\lambda}^1_1$, where $\bl_1^1 = (\lambda^1_1, \ti{\lambda}^1_1)$ is an arbitrary random variable pair in $(\mathbb{L}^2)^2$, with its corresponding state trajectory being $(X^{\bea_1,\bl_1^1}(\cdot), Y^{\bea_1,\bl_1^1}(\cdot), Z^{\bea_1,\bl_1^1}(\cdot))$. Moreover, let $(X^{\bea_1,\bl_1^\ep}(\cdot), Y^{\bea_1,\bl_1^\ep}(\cdot), Z^{\bea_1,\bl_1^\ep}(\cdot))$ denote the corresponding state trajectory for the perturbed variable pair $\bl_1^\ep$.
	
	To simplify notation, we replace the superscripts $(\bea_1,\bl_1^\ast)$, $(\bea_1,\bl_1^\ep)$, and $(\bea_1,\bl_1^1)$ of the state triple $(X^{\cdot}(\cdot),Y^{\cdot}(\cdot),Z^{\cdot}(\cdot))$ with $\ast$, $\ep$, and $1$, respectively.
	
	Then, we introduce the following variation equation:
	\begin{equation*}
		\left\{	\begin{aligned}
			\drm X^1 (t)=&\left[\cA_1X^1
			-\cB_1R_1^{-1}(\cB_1^{\top}Y^1
			+\cD_1^{\top}Z^1+\lambda^1_1)
			\right]\drm t\\
			&+[\cC_1X^1-\cD_1R_1^{-1}(\cB_1^{\top}Y^1+\cD_1^{\top}Z^1+\lambda^1_1) ]\drm W (t),\\				
			\drm Y^1 (t)=&-[\cA_1^{\top}Y^1
			+\cC_1^{\top}Z^1+Q_1X^1+\ti{\lambda}^1_1]\drm t
			+Z^1\drm W (t),\\
			X^1(0)=&0,\qquad Y^1(T)=G_1X^1(T).
		\end{aligned}\right.
	\end{equation*}
	Notice that
	\begin{equation}\label{eq:3.7-0}
		\begin{aligned}
			&\lim_{\epsilon\to 0}\frac{\hat{\hat{J}}^{\bea_1}(\la^\ep_1(\cdot),\ti{\la}^\ep_1(\cdot))-\hat{\hat{J}}^{\bea_1}(\la^\ast_1(\cdot),\ti{\la}^\ast_1(\cdot))}{\ep}\\
			=&2 \mathbb{E}\Big\{\int_0^T\big[\langle Q_1X^{\ast},X^1\rangle
			+ \langle \cB_1R_1^{-1}\cB_1^{\top}Y^{\ast},Y^1\rangle+ \langle \cD_1R_1^{-1}\cD_1^{\top}Z^{\ast},Z^1 \rangle
			+\langle \cD_1R_1^{-1}\cB_1^{\top}Y^{\ast},Z^1\rangle\\
			&
			+\langle \cD_1R_1^{-1}\cB_1^{\top}Y^1,Z^{\ast}\rangle-\langle \lambda_1^\ast,R_1^{-1}\lambda^1_1\rangle-\langle\lambda^1_1,\alpha_1\rangle+\langle \ti{\la}_1^\ast,X^1\rangle+\langle \ti{\la}^1_1,X^{\ast}\rangle-\langle \ti{\lambda}^1_1,\beta_1\rangle\big]\drm s\\
			&+\langle G_1X^{\ast}(T),X^1(T)\rangle\Big\}.
		\end{aligned}
	\end{equation}
	Applying It\^{o}'s formula to $\langle Y^{\ast},X^1\rangle$ yields
	\begin{equation*}
		\begin{aligned}
			\drm\langle Y^{\ast},X^1\rangle
			=&\big[\langle Y^{\ast},\cA_1X^1
			-\cB_1R_1^{-1}(\cB_1^{\top}Y^1
			+\cD_1^{\top}Z^1
			+\lambda^1_1)\rangle +\langle Z^{\ast} , \cC_1X^1
			-\cD_1R_1^{-1}(\cB_1^{\top}Y^1
			+\cD_1^{\top}Z^1
			+\lambda^1_1) \rangle\\
			&-\langle \cA_1^{\top}Y^{\ast}+\cC_1^{\top}Z^{\ast} +Q_1X^{\ast}+\ti{\lambda}^{\ast}_1,X^1\rangle \big] \drm s
			+\langle \cdots \rangle \drm W(s).
		\end{aligned}
	\end{equation*}
	By taking the expectation on both sides of the above equation, we have that
	\begin{equation}\label{eq:3.7-1}
		\begin{aligned}
			\mathbb{E}\langle G_1X^{\ast}(T),X^1(T)\rangle
			=&\mathbb{E}\bigg\{\int_{0}^{T}\big[\langle Y^{\ast},
			-\cB_1R_1^{-1}(\cB_1^{\top}Y^1
			+\cD_1^{\top}Z^1
			+\lambda^1_1)\rangle+\langle Z^{\ast} , 
			-\cD_1R_1^{-1}(\cB_1^{\top}Y^1
			+\cD_1^{\top}Z^1
			+\lambda^1_1) \rangle\\
			&-\langle  Q_1X^{\ast}+\ti{\lambda}^{\ast}_1,X^1\rangle \big]\drm s \bigg\}.
		\end{aligned}
	\end{equation}
	
	By substituting \eqref{eq:3.7-1} into \eqref{eq:3.7-0}, together with the expression given in \eqref{eq:OC1}, we obtain  
	\begin{equation*}
		\begin{aligned}
			0=&\lim_{\epsilon\to 0}\frac{\hat{\hat{J}}^{\bea_1}(\lambda^{\epsilon}_1(\cdot),\ti{\lambda}^{\epsilon}_1(\cdot))
				-\hat{\hat{J}}^{\bea_1}(\lambda^{\ast}_1(\cdot),\ti{\lambda}^{\ast}_1(\cdot))}{\epsilon}\\
			=&2 \mathbb{E}\Big\{\int_0^T\big[\langle -R_1^{-1}\left[\cB_1^\ast Y^{\ast}+\cD_1^\ast Z^{\ast}+\lambda_1^\ast\right]-\alpha_1, \lambda^1_1\rangle+\langle \ti{\la}^1_1,X^{\ast}\rangle-\langle \ti{\la}^1_1,\beta_1\rangle\big]\drm s \Big\}.
		\end{aligned}
	\end{equation*}
	
	Therefore, by the arbitrariness of the variation pair $(\lambda^1_1, \ti{\la}^1_1)$, we get that if $(\lambda^{\ast}_1,\ti{\la}_1^{\ast})$ is the optimal pair, then $\mathbb{E}\tu_1^{\bea_1,\bl_1^\ast}=\alpha_1$ and $\mathbb{E}X^{\ast}=\beta_1$. 
\end{proof}

Now, we turn to proving the main theorem for Problem (F-3) in detail. First, we provide the detailed proof of Lemma \ref{lem:cc}.

\begin{proof}[The proof of Lemma \ref{lem:cc}]
	By inserting the operator representations of \(\tu_1^{\bea_1,\bl_1}(\cdot)\), \(X^{\bea_1,\bl_1}(\cdot)\), \(X^{\bea_1,\bl_1}(T)\), and \(\beta_1(T)\) , which are given from \eqref{eq:ou1} to \eqref{eq:bo} respectively, into the cost functional \eqref{eq:cost3}, we obtain that
\begin{align*}
	&\tilde{J}_1(\alpha_1(\cdot),\beta_1(\cdot))\\
	&=\mathbb{E}\Big\{\int_0^T\Big[\langle Q_1 (\cK_{2,1}x+\cK_{2,2}\alpha_1+\cK_{2,3}\beta_1+\cK_{2,4}u_2+\cK_{2,5}),\cK_{2,1}x+\cK_{2,2}\alpha_1+\cK_{2,3}\beta_1+\cK_{2,4}u_2+\cK_{2,5}\rangle\\
	&+\langle\bar{Q}_1 
	\beta_1 ,\beta_1 \rangle+ \langle R_1 (\cK_{1,1} x+\cK_{1,2} \alpha_1+\cK_{1,3}\beta_1+\cK_{1,4} u_2+\cK_{1,5}) ,\cK_{1,1} x+\cK_{1,2} \alpha_1+\cK_{1,3}\beta_1+\cK_{1,4} u_2+\cK_{1,5} \rangle\\
	&+\langle \bar{R}_1 \alpha_1 ,\alpha_1 \rangle\Big]\drm s+\langle G_1(\cK_{3,1}x+\cK_{3,2}\alpha_1+\cK_{3,3}\beta_1+\cK_{3,4}u_2+\cK_{3,5}),\cK_{3,1}x+\cK_{3,2}\alpha_1+\cK_{3,3}\beta_1+\cK_{3,4}u_2+\cK_{3,5}\rangle\Big\}\\
	&=\left\langle(\cK_{2,1}^\ast Q_1\cK_{2,1}+\cK_{1,1}^\ast R_1\cK_{1,1}+\cK_{3,1}^\ast  G_1 \cK_{3,1})x,x\right\rangle_{\RR^n}\\
	&+\left\langle(\cK_{2,2}^\ast Q_1\cK_{2,2}+\cK_{1,2}^\ast R_1\cK_{1,2}+\bar{R}_1+\cK_{3,2}^\ast  G_1 \cK_{3,2})\alpha_1,\alpha_1\right\rangle_{\LL^2}\\
	&+\left\langle(\cK_{2,3}^\ast Q_1\cK_{2,3}+\cK_{1,3}^\ast R_1\cK_{1,3}+\bar{Q}_1+\cK_{3,3}^\ast  G_1 \cK_{3,3})\beta_1,\beta_1\right\rangle_{\LL^2}\\
	&+\left\langle(\cK_{2,4}^\ast Q_1\cK_{2,4}+\cK_{1,4}^\ast R_1\cK_{1,4}+\cK_{3,4}^\ast  G_1 \cK_{3,4})u_2,u_2\right\rangle_{\cU_2}\\
	&+2\left\langle (\cK_{2,2}^\ast Q_1\cK_{2,1}+\cK_{1,2}^\ast R_1 \cK_{1,1}+\cK_{3,2}^\ast  G_1 \cK_{3,1})x,\alpha_1\right\rangle_{\LL^2}\\
	&+2\left\langle (\cK_{2,3}^\ast Q_1\cK_{2,1}+\cK_{1,3}^\ast R_1 \cK_{1,2}+\cK_{3,3}^\ast  G_1 \cK_{3,1})x,\beta_1\right\rangle_{\LL^2}\\
	&+2\left\langle (\cK_{2,4}^\ast Q_1\cK_{2,1}+\cK_{1,4}^\ast R_1 \cK_{1,2}+\cK_{3,4}^\ast  G_1 \cK_{3,1})x,u_2\right\rangle_{\cU_2}\\
	&+2\left\langle (\cK_{2,2}^\ast Q_1\cK_{2,3}+\cK_{1,2}^\ast R_1 \cK_{1,3}+\cK_{3,2}^\ast  G_1 \cK_{3,3})\beta_1,\alpha_1\right\rangle_{\LL^2}\\
	&+2\left\langle (\cK_{2,2}^\ast Q_1\cK_{2,4}+\cK_{1,2}^\ast R_1 \cK_{1,4}+\cK_{3,2}^\ast  G_1 \cK_{3,4})u_2,\alpha_1\right\rangle_{\LL^2}\\
	&+2\left\langle (\cK_{2,3}^\ast Q_1\cK_{2,4}+\cK_{1,3}^\ast R_1 \cK_{1,4}+\cK_{3,3}^\ast  G_1 \cK_{3,4})u_2,\beta_1\right\rangle_{\LL^2}\\
	&+2\left\langle x,  \cK_{2,1}^\ast Q_1 \cK_{2,5}+\cK_{1,1}^\ast R_1 \cK_{1,5}+\cK_{3,1}^\ast G_1 \cK_{3,5}\right\rangle_{\RR^n}\\
	&+2\left\langle \alpha_1,  \cK_{2,2}^\ast Q_1 \cK_{2,5}+\cK_{1,2}^\ast R_1 \cK_{1,5}+\cK_{3,2}^\ast G_1 \cK_{3,5}\right\rangle_{\LL^2}\\
	&+2\left\langle \beta_1,  \cK_{2,3}^\ast Q_1 \cK_{2,5}+\cK_{1,3}^\ast R_1 \cK_{1,5}+\cK_{3,3}^\ast G_1 \cK_{3,5}\right\rangle_{\LL^2}\\
	&+2\left\langle u_2,  \cK_{2,4}^\ast Q_1 \cK_{2,5}+\cK_{1,4}^\ast R_1 \cK_{1,5}+\cK_{3,4}^\ast G_1\cK_{3,5}\right\rangle_{\LL^2}\\
	&+2\left\langle  \cK_{2,5}^\ast Q_1 \cK_{2,5}+\cK_{1,5}^\ast R_1 \cK_{1,5}+\cK_{3,5}^\ast G_1 \cK_{3,5},\right.\\
	&\left.\qquad\quad  \cK_{2,5}^\ast Q_1 \cK_{2,5}+\cK_{1,5}^\ast R_1 \cK_{1,5}+\cK_{3,5}^\ast G_1  \cK_{3,5}\right\rangle_{\LL^2}.
\end{align*}
Based on assumptions (H1) and (H2), we obtain that
\begin{align*}
	&\left\langle(\cK_{2,2}^\ast Q_1\cK_{2,2}+\cK_{1,2}^\ast R_1\cK_{1,2}+\bar{R}_1+\cK_{3,2}^\ast  G_1 \cK_{3,2})\alpha_1,\alpha_1\right\rangle_{\LL^2}\geq \delta \EE\int_0^T |\alpha_1|^2\drm s>0,\\
	&\left\langle(\cK_{2,3}^\ast Q_1\cK_{2,3}+\cK_{1,3}^\ast R_1\cK_{1,3}+\bar{Q}_1+\cK_{3,3}^\ast  G_1 \cK_{3,3})\beta_1,\beta_1\right\rangle_{\LL^2}\geq \delta \EE\int_0^T |\beta_1|^2\drm s>0,
\end{align*}
which implies the strict convexity of the cost functional \(\ti{J}_1(\alpha_1(\cdot),\beta_1(\cdot))\) with respect to \(\alpha_1(\cdot)\) and \(\beta_1(\cdot)\) respectively.
\end{proof}

\begin{proof}[The proof of Theorem \ref{th:3.10}] Suppose that \((\al_1^\ast, \beta_1^\ast)\) are the optimal control variables. Then we have that
\begin{align*}
	&\tilde{J}_1(\alpha_1^\ast(\cdot),\beta_1^\ast(\cdot))\\
	&=\left\langle(\cK_{2,1}^\ast Q_1\cK_{2,1}+\cK_{1,1}^\ast R_1\cK_{1,1}+\cK_{3,1}^\ast  G_1 \cK_{3,1})x,x\right\rangle_{\RR^n}\\
	&+\left\langle(\cK_{2,2}^\ast Q_1\cK_{2,2}+\cK_{1,2}^\ast R_1\cK_{1,2}+\bar{R}_1+\cK_{3,2}^\ast  G_1 \cK_{3,2})\alpha_1,\alpha_1\right\rangle_{\LL^2}\\
	&+\left\langle(\cK_{2,3}^\ast Q_1\cK_{2,3}+\cK_{1,3}^\ast R_1\cK_{1,3}+\bar{Q}_1+\cK_{3,3}^\ast  G_1 \cK_{3,3})\beta_1,\beta_1\right\rangle_{\LL^2}\\
	&+\left\langle(\cK_{2,4}^\ast Q_1\cK_{2,4}+\cK_{1,4}^\ast R_1\cK_{1,4}+\cK_{3,4}^\ast  G_1 \cK_{3,4})u_2,u_2\right\rangle_{\cU_2}\\
	&+2\left\langle (\cK_{2,2}^\ast Q_1\cK_{2,1}+\cK_{1,2}^\ast R_1 \cK_{1,1}+\cK_{3,2}^\ast  G_1 \cK_{3,1})x,\alpha_1\right\rangle_{\LL^2}\\
	&+2\left\langle (\cK_{2,3}^\ast Q_1\cK_{2,1}+\cK_{1,3}^\ast R_1 \cK_{1,2}+\cK_{3,3}^\ast  G_1 \cK_{3,1})x,\beta_1\right\rangle_{\LL^2}\\
	&+2\left\langle (\cK_{2,4}^\ast Q_1\cK_{2,1}+\cK_{1,4}^\ast R_1 \cK_{1,2}+\cK_{3,4}^\ast  G_1 \cK_{3,1})x,u_2\right\rangle_{\cU_2}\\
	&+2\left\langle (\cK_{2,2}^\ast Q_1\cK_{2,3}+\cK_{1,2}^\ast R_1 \cK_{1,3}+\cK_{3,2}^\ast  G_1 \cK_{3,3})\beta_1,\alpha_1\right\rangle_{\LL^2}\\
	&+2\left\langle (\cK_{2,2}^\ast Q_1\cK_{2,4}+\cK_{1,2}^\ast R_1 \cK_{1,4}+\cK_{3,2}^\ast  G_1 \cK_{3,4})u_2,\alpha_1\right\rangle_{\LL^2}\\
	&+2\left\langle (\cK_{2,3}^\ast Q_1\cK_{2,4}+\cK_{1,3}^\ast R_1 \cK_{1,4}+\cK_{3,3}^\ast  G_1 \cK_{3,4})u_2,\beta_1\right\rangle_{\LL^2}\\
	&+2\left\langle x,  \cK_{2,1}^\ast Q_1 \cK_{2,5}+\cK_{1,1}^\ast R_1 \cK_{1,5}+\cK_{3,1}^\ast G_1 \cK_{3,5}\right\rangle_{\RR^n}\\
	&+2\left\langle \alpha_1,  \cK_{2,2}^\ast Q_1 \cK_{2,5}+\cK_{1,2}^\ast R_1 \cK_{1,5}+\cK_{3,2}^\ast G_1 \cK_{3,5}\right\rangle_{\LL^2}\\
	&+2\left\langle \beta_1,  \cK_{2,3}^\ast Q_1 \cK_{2,5}+\cK_{1,3}^\ast R_1 \cK_{1,5}+\cK_{3,3}^\ast G_1 \cK_{3,5}\right\rangle_{\LL^2}\\
	&+2\left\langle u_2,  \cK_{2,4}^\ast Q_1 \cK_{2,5}+\cK_{1,4}^\ast R_1 \cK_{1,5}+\cK_{3,4}^\ast G_1\cK_{3,5}\right\rangle_{\LL^2}\\
	&+2\left\langle  \cK_{2,5}^\ast Q_1 \cK_{2,5}+\cK_{1,5}^\ast R_1 \cK_{1,5}+\cK_{3,5}^\ast G_1 \cK_{3,5},\right.\\
	&\left.\qquad\quad  \cK_{2,5}^\ast Q_1 \cK_{2,5}+\cK_{1,5}^\ast R_1 \cK_{1,5}+\cK_{3,5}^\ast G_1  \cK_{3,5}\right\rangle_{\LL^2}.
\end{align*}
Therefore, \((\alpha^\ast_1(\cdot),\beta_1^\ast(\cdot))\) is the optimal pair if and only if 
\begin{equation}
	\left\{
	\begin{aligned}
		&(\cK_{2,2}^\ast Q_1\cK_{2,2}+\cK_{1,2}^\ast R_1\cK_{1,2}+\bar{R}_1+\cK_{3,2}^\ast  G_1 \cK_{3,2})\alpha_1^\ast+(\cK_{2,2}^\ast Q_1\cK_{2,1}+\cK_{1,2}^\ast R_1 \cK_{1,1}+\cK_{3,2}^\ast  G_1 \cK_{3,1})x\\
		&\quad+(\cK_{2,2}^\ast Q_1\cK_{2,3}+\cK_{1,2}^\ast R_1 \cK_{1,3}+\cK_{3,2}^\ast  G_1 \cK_{3,3})\beta_1^\ast+(\cK_{2,2}^\ast Q_1\cK_{2,4}+\cK_{1,2}^\ast R_1 \cK_{1,4}+\cK_{3,2}^\ast  G_1 \cK_{3,4})u_2\\
		&\quad+\cK_{2,2}^\ast Q_1 \cK_{2,5}+\cK_{1,2}^\ast R_1 \cK_{1,5}+\cK_{3,2}^\ast G_1  \cK_{3,5}=0,\\
		&(\cK_{2,3}^\ast Q_1\cK_{2,3}+\cK_{1,3}^\ast R_1\cK_{1,3}+\bar{Q}_1+\cK_{3,3}^\ast  G_1 \cK_{3,3})\beta_1^\ast+ (\cK_{2,3}^\ast Q_1\cK_{2,1}+\cK_{1,3}^\ast R_1 \cK_{1,1}+\cK_{3,3}^\ast  G_1 \cK_{3,1})x\\
		&\quad+(\cK_{2,3}^\ast Q_1\cK_{2,2}+\cK_{1,3}^\ast R_1 \cK_{1,2}+\cK_{3,3}^\ast  G_1 \cK_{3,2})\alpha_1^\ast+(\cK_{2,3}^\ast Q_1\cK_{2,4}+\cK_{1,3}^\ast R_1 \cK_{1,4}+\cK_{3,3}^\ast  G_1 \cK_{3,4})u_2\\
		&\quad+\cK_{2,3}^\ast Q_1 \cK_{2,5}+\cK_{1,3}^\ast R_1 \cK_{1,5}+\cK_{3,3}^\ast G_1  \cK_{3,5}=0.
	\end{aligned}
	\right.
\end{equation}
Moreover, we can equivalently rewrite the above system of equations with respect to the variables \(\alpha_1^\ast\) and \(\beta_1^\ast\) into the following matrix equation
\begin{align} (\cW+\cK_{23}^\ast\cT\cK_{23})\cdot(\alpha^\ast_1,\beta_1^\ast)^{\top}+\cK_{23}^\ast\cT \cK_{14}\cdot (x,u_2)^{\top}+\cK_{23}^\ast\cT\cK_5=(0,0)^{\top},
\end{align}
where \(\cW=\begin{pmatrix}
	\bar{R}_1&0\\
	0&\bar{Q}_1
\end{pmatrix}\),  \(\cK_{23}=\begin{pmatrix}
	\cK_{1,2}&\cK_{1,3}\\
	\cK_{2,2}&\cK_{2,3}\\
	\cK_{3,2}&\cK_{3,3}
\end{pmatrix}\), \(\cT=\begin{pmatrix}
	R_1&0&0\\
	0&Q_1&0\\
	0&0&G_1
\end{pmatrix}\), \(\cK_{14}=\begin{pmatrix}
	\cK_{1,1}&\cK_{1,4}\\
	\cK_{2,1}&\cK_{2,4}\\
	\cK_{3,1}&\cK_{3,4}
\end{pmatrix}\),  and \(\cK_5=
(\cK_{1,5},\cK_{2,5},\cK_{3,5})^\ast\).
\end{proof}
 
\section{The Proof of Proposition \ref{prop:alm_bound}}\label{app:A}
\begin{proof}
	Fix an iteration $p$ such that $\rho_{v,i}^{(p)}>\eta^{-1}$. Then, rearranging~\eqref{eq:eps_net_def} gives
	\[r_{v,i}^{(p)}=\frac{\lambda_{v,i}^{(p+1)}-\lambda_{v,i}^{(p)}-\varepsilon_{\rm net}^{(p)}}{\rho_{v,i}^{(p)}}.\]
	
	Hence, by the triangle inequality, we have
	\begin{equation}\label{eq:R_isolated}
		\mathcal R_{v,i}^{(p)}=\|r_{v,i}^{(p)}\|\le\frac{\|\lambda_{v,i}^{(p+1)}-\lambda_{v,i}^{(p)}\|}{\rho_{v,i}^{(p)}}+\frac{\|\varepsilon_{\rm net}^{(p)}\|}{\rho_{v,i}^{(p)}}.
	\end{equation}
	
	Now, we consider the proximal dual objective at iteration $p$, with proximal centre chosen as the current iterate,i.e.\ $\lambda_{v,i}^{\rm prev}=\lambda_{v,i}^{(p)}$. Then from \eqref{eq:dual_update}, we obtain
	\[\nabla_\lambda \mathcal L_\lambda^{(p)}\bigl(\phi\bigr)
	=-r_{v,i}^{(p)}+\eta(\lambda_{v,i}(\cdot,\phi)-\lambda_{v,i}^{\rm prev})=-r_{v,i}^{(p)},\] with $r_{v,i}^{(p)}=\text{\rm viol}$.  Let $\lambda_{v,i}^{*,(p)} := \arg\min_\lambda \mathcal
	L_\lambda^{(p)}(\lambda)$ denote the exact minimiser of the dual subproblem at iteration $p$.
	
	Because the inclusion of the proximal term renders $\mathcal{L}_\lambda^{(p)}$ strictly $\eta$-strongly convex in $\lambda_{v,i}$, then we have
	\[\|\lambda_{v,i}^{(p)} - \lambda_{v,i}^{*,(p)}\| \;\le\; \eta^{-1} \bigl\| \nabla_\lambda \mathcal{L}_\lambda^{(p)}(\lambda^{(p)}) \bigr\| \;=\; \eta^{-1} \|{-r_{v,i}^{(p)}}\| \;=\; \eta^{-1}\mathcal{R}_{v,i}^{(p)}.
	\]
	Applying the triangle inequality to the actual dual increment, and invoking the subproblem bound from Assumption~\ref{ass:bounded_errors}, we establish:
	\begin{align}\label{eq:lambda_inc}
		\|\lambda_{v,i}^{(p+1)} - \lambda_{v,i}^{(p)}\| &\;\le\; \|\lambda_{v,i}^{(p)} - \lambda_{v,i}^{*,(p)}\| + \|\lambda_{v,i}^{(p+1)} - \lambda_{v,i}^{*,(p)}\| \notag \\
		&\;\le\; \eta^{-1}\mathcal{R}_{v,i}^{(p)} + \bar\varepsilon_{\rm opt}.
	\end{align}
	
	Substituting~\eqref{eq:lambda_inc} into~\eqref{eq:R_isolated}, we obtain
	\[\mathcal R_{v,i}^{(p)}\le\frac{\eta^{-1}}{\rho_{v,i}^{(p)}}\mathcal R_{v,i}^{(p)}+\frac{\varepsilon_{\rm opt}^{(p)}+\|\varepsilon_{\rm net}^{(p)}\|}{\rho_{v,i}^{(p)}}.
	\]
	Since $\rho_{v,i}^{(p)}>\eta^{-1}$, rearranging yields
	\[\mathcal R_{v,i}^{(p)}\le\frac{\varepsilon_{\rm opt}^{(p)}+\|\varepsilon_{\rm net}^{(p)}\|}{\rho_{v,i}^{(p)}-\eta^{-1}}.
	\]
	Applying Assumption~\ref{ass:bounded_errors} gives~\eqref{eq:viol_bound}. 
	
	The final claim follows immediately from the adaptive penalty design: since the scheme multiplies the penalty by a strict factor $\tau > 1$ whenever the violation stagnates above the specified tolerance, the sequence $\rho_{v,i}^{(p)}$ either terminate with a satisfied tolerance or diverge to infinity (i.e., $\rho_{v,i}^{(p)}\to\infty$). In the latter case, the upper bound in \eqref{eq:viol_bound} is strictly driven to zero.
\end{proof}
\end{document}